\documentclass[article, 12pt]{amsart}

\usepackage{amsmath,amssymb,amsthm,graphicx}
\usepackage{latexsym}
\usepackage{amscd}
\usepackage{amsfonts}
\usepackage[usenames,dvipsnames]{color}
\usepackage{verbatim}
\usepackage{float}
\usepackage{url}
\usepackage{xypic}
\usepackage{marginnote}\usepackage{mathtools}
\usepackage[pdftex,colorlinks]{hyperref}

\DeclareMathOperator{\BBS}{BBS}

\DeclareMathOperator{\Imag}{Imag}

\newtheorem{prop}{Proposition}[section]
\newtheorem{theorem}[prop]{Theorem}

\newtheorem{lem}[prop]{Lemma}

\theoremstyle{definition}

\newtheorem{definition}[prop]{Definition}
\newtheorem{rem}[prop]{Remark}
\newtheorem{example}[prop]{Example}

\newtheorem*{rem*}{Remark}
\newtheorem*{notation*}{Notation}

\numberwithin{equation}{section}

\newcommand{\ra}{\rightarrow}
\newcommand{\bC}{\mathbb{C}}
\newcommand{\bP}{\mathbb{P}}
\newcommand{\bH}{\mathbb{H}}

\DeclareMathOperator{\pt}{pt}
\DeclareMathOperator{\Fl}{Fl}
\DeclareMathOperator{\Gr}{Gr}

\newcommand{\Hom}{\mathrm{Hom}}

\newcommand{\mwG}{\widehat{\mathcal{G}}}

\begin{document}

\title[Bioriented flags and resolutions]{Bioriented flags and resolutions of  Schubert varieties}
\author{Daniel Cibotaru}
\thanks{This was partially supported by a BITDEFENDER post-doctoral fellowship}
\address{Departamento de Matem\'atica, Universidade Federal do Cear\'a}
\email{daniel@mat.ufc.br}
\begin{abstract} We use  incidence relations running in two directions in order to construct  a Kempf-Laksov type  resolution for any Schubert variety of the complete \emph{flag manifold} but also an \emph{embedded resolution} for any Schubert variety in the Grassmannian. These constructions are alternatives to the celebrated Bott-Samelson resolutions. The second process led to the introduction of $W$-flag varieties, algebro-geometric objects that interpolate between the standard flag manifolds and products of Grassmannians, but which are singular in general. The surprising simple desingularization of a particular such type of variety produces an embedded resolution of the Schubert variety within the Grassmannian.
\end{abstract}
\maketitle
\thispagestyle{empty}
\tableofcontents

\section{Introduction}
The Kempf-Laksov resolutions of the Schubert varieties in the Grassmannian are well-known since the 70's and were used to prove the celebrated determinantal formula \cite{KL}. These objects are very easy to describe via \emph{linear} incidence relations and quite practical for many purposes. 
We call the relations linear because they  involve a total order relation.  One of the purposes we pursue in this article is to produce  a similar" desingularization but for the 
(generalized) Schubert varieties in the flag manifold associated to the elements of the Weyl group of $GL_n(\bC)$, i.e. permutations. In other words, we aim for a resolution  that does not use the Borel orbits and the fiber products of parabolic groups as in the original Bott-Samelson construction \cite{BS,De,Ha}. It turns out that can be done quite simply and naturally, but there is  a price to pay. The resolutions  are subspaces of a matrix-product of  Grassmannian spaces. The elements of these subspaces satisfy incidence relations both in the vertical and the horizontal directions with respect to the matrix display of the product of Grassmannians. We call such objects bioriented flags. There is quite a bit of redundancy in the definition of a bioriented flag, but it is this redundancy that accounts for their simple description.  

In \cite{Ma1}, Magyar gave a different construction of the Bott-Samelson resolutions, realizing them as subspaces of a product of Grassmannians. A particularly explicit description of Magyar's construction for $GL_n(\bC)$ using incidence relations  was given by Jones and Woo in \cite{JW}. One common feature of all Bott-Samelson resolutions of a fixed Schubert variety determined by a permutation $w$ is that they use a total of $l(w)$ Grassmannian spaces. After eliminating redundancy, this coincides with the number of Grassmannians needed in the resolution via bioriented flags. It is then less than a surprise that the bioriented flag resolution is isomorphic to the Bott-Samelson resolution corresponding to the "bubblesort" reduced word decomposition of the permutation that defines the Schubert variety.  We prove this using \cite{JW}.

The other point of interest in this article is to look for an \emph{embedded resolution} of a Schubert variety within the Grassmannian. There  are certain common points with the previous case. Again, the Bott-Samelson resolution can be used to produce such a resolution by taking the Grassmanian to  be the Schubert variety itself inside the flag manifold $G/B$ together with the projection $G/B\ra G/P$. We propose an alternative, direct  construction in the same spirit to what is done in the first part of the article. We use again a bioriented flag, but the incidence relations have something new to them. Rather than just inclusions of the smaller dimensional spaces into the bigger ones, one has to deal now with the inclusions of the smaller dimensional spaces into the sum of the bigger dimensional spaces and a fixed space. This is in general a source of singularities.  We distill a definition for such objects, henceforth called "W-flag varieties". 

Let us  describe in a few words the construction of the embedded resolution. It starts with the well-known observation that an open subset of the  regular part of a Schubert variety can be seen as the stable manifold associated to a critical manifold for a certain gradient flow of a Morse-Bott function on the Grassmannian \cite{HL,Wo}. As such, it has a companion Schubert variety  which plays the role of the closure of the unstable manifold for the same flow. The unstable manifold  fibers over a product of projective spaces which is  a critical manifold of the Morse-Bott function. All this is  valid in the $C^{\infty}$ category.

In order to get the holomorphic resolution, the following process is used. Fix any fiber of the companion Schubert variety with respect  to the projection to the critical manifold. It plays the role of the  variety of "normal directions" in which we seek to deform the original Schubert variety in order to "cover"  the Grassmannian, a posteriori, with families of Kempf-Laksov resolutions.  More precisely to each such "normal direction" we associate in a one-to-one fashion a  partial flag of the same type as  the original Schubert variety, obtaining thus a family of Schubert varieties of the same type. Unfortunately,  this algebraic family of partial flags is not a proper subvariety, owing to the fact that the fiber is not compact. Hence the union of   Schubert varieties corresponding to this family of flags covers only an open dense set of the ambient Grassmannian.   

One then needs to compactify  \emph{the embedding} of the (non-compact) fiber in the mentioned space of partial flags. The emphasis is on the map, rather than on the set. By this we mean that we are looking for a proper variety together with a map to the space of partial flags that extends the embedding of the fiber in the space of partial flags of the previous paragraph.  The "compactification" is  a constrained $W$-flag  variety, definition to be given momentarily. These are, in general, singular algebraic objects and do not seem to have  simple desingularizations. But, quite surprisingly, the very particular type of incidence relations describing \emph{the compactification} under inspection  lead to the construction of a simple resolution in which  double incidence relations are present. In other words, bioriented flags show up again. To finish up, each point in the resolution of the compactification of the fiber determines a partial flag. This flag can be used to construct a Kempf-Laksov resolution for the corresponding Schubert variety as one naturally has a projection to the original Grassmannian. Doing this for all points in the resolution of the compactification we actually get a birational map to the Grassmannian, whose restriction to a certain submanifold gives a resolution of the Schubert variety we started with.

We give now the working definitions in the analytic category.
\begin{definition}\label{defo1} Let $X$ be a complex space. Then a (direct) resolution/desingularization of $X$ is a non-singular complex space $\tilde{X}$ together with a proper, analytic, birational map $\pi:\tilde{X}\ra X$.

Let $X$ be a proper analytic subvariety  of a regular complex space $Y$. Then an embedded resolution of $X$ is a pair $(\tilde{Y},\tilde{X})$ of regular spaces with $\tilde{X}$ a subspace of $\tilde{Y}$ together with a proper, analytic map $\pi:\tilde{Y}\ra Y$ such that
\begin{itemize}
\item[(a)] $\pi$ is birational;
\item[(b)] $\pi\bigr|_{\tilde{X}}:\tilde{X}\ra X$ is birational
\item[(c)] there exists a proper analytic subvariety $X'\subset X$ such that $\pi^{-1}(X\setminus X')\subset \tilde{X}$.
\end{itemize}
\end{definition}
\begin{rem} First, notice that we do not demand more specific knowledge of the exceptional divisors, as is customary for what is called strong desingularization \cite{Wl}, where one requires that the exceptional divisor is a normal crossings divisor.

Second, it is unreasonable to expect that $\pi^{-1}(X)\subset \tilde{X}$ as simple examples of resolutions of singularities show. Item (c) is included as the reasonable substitute condition.
\end{rem}

Let $E$ be a complex vector space of dimension $n$, let $\Fl_{1,\ldots, n}(E)$ be the complete flag manifold of subspaces of $E$. For every fixed $F_*$ and every $w\in S_n$, the symmetric group on $n$-elements, one has a Schubert cell $S^w(F)$ defined via (see \cite{Fu} page 157)
\[ S^w(F):=\{\ell_*\in\Fl_{1,\ldots, n}(E)~|~\dim{\ell_p\cap F_q}=\#\{i\in\{1,\ldots, n\}~|~i\leq p,~ w(i)\leq q\}
\]
The rank function 
\begin{equation} \label{Beq1} D^w(p,q):= \#\{i~|~i\leq p, ~w(i)\leq q\}=\#w(\{1,2\ldots,p\})\cap \{1,\ldots,q\}\end{equation}
 determines the Bruhat order on $S_n$ (see \cite{Fu}, page 173). To it we associate the product space
 \begin{equation}\mathcal{P}^w:=\prod_{i,j=1}^n\Gr_{D^w_{i,j}}(E)
 \end{equation}
 where if $D^w_{i,j}=0$ or $n$  we get point spaces. Define the following bioriented flag:
\[ \Fl^w(E):=\{\ell\in \mathcal{P}^w~|~\ell_{i,j}\subset \ell_{i,j+1},~~\ell_{i,j}\subset \ell_{i+1,j}\}
\]
The incidence relations make sense as the rank matrix (\ref{Beq1}) is increasing along every line and along every column. The next statement puts together Theorem \ref{Flres1} and Theorem \ref{BBS}.
\begin{theorem} Let $\hat{S}^w(F):=\{\ell\in \Fl{^w}(E)~|~\ell_{n,i}=F_i, ~i=1,\ldots, n\}$. Then $\hat{S}^w(F)$ is a manifold of dimension equal to the number of inversions of $w$ and the projection 
\[\hat{S}^w(F)\ra \Fl_{1,\ldots,n}(E),\qquad \ell_{*,*}\ra \ell_{*,n}\]
is a resolution of $S^w(F)$ which is isomorphic to the Bott-Samelson resolution of $S^w(F)$ corresponding to the bubblesort reduced word presentation of $w$.
\end{theorem}

It would be interesting to obtain an analogous  description for all Bott-Samelson resolutions of $S^w(F)$, corresponding to other presentations of $w$ as a reduced word.

\vspace{0.1cm}

We turn our attention to the Grassmannian $\Gr_k(E)$ and fix a strictly increasing sequence
\[ 1\leq \beta_1<\beta_2<\ldots< \beta_k\leq n,
\]
henceforth called multi-index. Define the Schubert variety $\overline{V_{\beta}}=\{L\in \Gr_k(E)~|~\dim{L\cap F_{\beta_i}}\geq i\}$. We reserve the notation $V_{\beta}$ for the subset where the dimensional condition is replaced by an equality. Choose a complementary flag $G^\beta$ such that $F_{\beta_i}\oplus G^{\beta_i}= E$ and let $F_{\beta_i}^{\beta_{i-1}}:=F_{\beta_i}\cap G^{\beta_{i-1}}$ and
\[\bP=\prod_{i=1}^k\bP(F_{\beta_i}^{\beta_{i-1}})
\]
Define also 
\[{V^*_{\beta}}:=\{L\in\Gr_k(E)~|~\dim{L\cap G^{\beta_i}}=k-i,~ i=1,\ldots k\}\] with corresponding $\overline{V_{\beta}^*}$.
\begin{theorem} The sets $V_{\beta}$ and $V_{\beta}^*$  are  both diffeomorphic to  vector bundles over $\bP$, namely  to
\[ \bH:=\prod_{i=2}^k\Hom\left(\tau_i,\bigoplus_{j=1}^{i-1}\tau_j^{\perp}\right)\quad\mbox{and}\quad \bH^*:=\prod_{i=1}^k\Hom\left(\tau_i, \bigoplus_{j=i+1}^{k+1} \tau_j^{\perp}\right),
\]
where $\tau_i\ra \bP$ represent the pull-backs of the tautological bundles over $\bP(F_{\beta_i}^{\beta_{i-1}})$ while $\tau_i^{\perp}$ are their complements  within the trivial $\underline{F_{\beta_i}^{\beta_{i-1}}}$ and $\tau_{k+1}^{\perp}:=G^{\beta_k}$. The diffeomorphisms are  explicit and are both induced by taking sums of graphs of linear morphisms.
\end{theorem}
We fix a fiber $\bH^*_0$ of $\bH^*\ra \bP$, to wit
\[\bH^*_0:=\prod_{i=1}^k\Hom\left(L_i,\sum_{j=i+1}^{k+1}L_j^{\perp}\right)
\]
for fixed choices of lines $L_i\subset F_{\beta_i}^{\beta_{i-1}}$ and complements $L_i^{\perp}\subset F_{\beta_i}^{\beta_{i-1}}$ (naturally, $L_{k+1}^{\perp}:=G^{\beta_k}$). 

If $\Fl_{\beta_1,\ldots,\beta_k}(E)$ is the space of flags of subspaces having dimensions $\beta_1,\beta_2,\ldots, \beta_k$ then there exists a natural embedding
\begin{equation}\label{Hoeq1}\bH^*_0\ra\Fl_{\beta_1,\ldots,\beta_k}(E),\qquad (A_1,\ldots, A_k)\ra\left(\ldots, \sum_{j=1}^{i}\left(\Gamma_{A_j}+L_j^{\perp}\right),\ldots \right)
\end{equation}
 
We  look for a compactification of this embedding. The object that we get has the following structure.
\begin{definition}Let $W_2,\ldots W_k$ be vector subspaces of $E$ and let $1\leq a_1 <a_2 <...<a_k \leq  n$ be natural numbers. Then a subset $S$ of $\prod_{i=1}^k \Gr_{a_i}(E)$ described by the following type of incidence relations:
\[S=\{(\ell_1,\ldots,\ell_k)~|~\ell_i\subset \ell_{i+1}+W_{i+1}, ~i\leq k-1\}\]
is called a $W_*$-flag variety. Moreover if each $\ell_i$, $1\leq i\leq k$ is constrained to lie in a certain vector subspace $V_i \subset E$ such that $a_i\leq \dim{V_i}$ one talks about a  constrained $W$-flag variety.
\end{definition}
\begin{example} Notice that for $W_i\equiv 0,~\forall i$ one recovers the definition of a flag manifold while for $W_i\equiv E,~\forall i$ one gets the product of Grassmannians. 

The Kempf-Laksov resolution  (see \ref{KL1} below) of a Schubert variety in the Grassmannian  is an  $F_*$-constrained $W$-flag variety, where $W=0$ and $F_*$ is the partial flag.
\end{example}

It turns out that in our case, $\bH^*_0$ compactifies to the  constrained $W$-flag variety 
\[ \mathcal{G}\subset \prod_{i=1}^k\Gr_i(V_i^i) \quad\mbox{ where}\]
\begin{equation}\label{cons1} W_{i}:=L_{i}^{\perp},\quad \forall~2\leq i\leq k, \mbox{ and }\;\; V_i^i=\sum_{j=1}^iL_j+\sum_{j=i+1}^{k+1}L_j^{\perp}.
\end{equation}
\begin{rem} Deformations of flag varieties have been considered before, we mention here  the work of Feigin \& all \cite{CFFFR, F}. They are related to $W$-flag varieties but they are not the same. For example, for a map $f:E\ra E$ which is the projection onto a space $W^{\perp}$ induced by a decomposition $E=W\oplus W^{\perp}$, a condition as in \cite{CFFFR} of type $f(\ell_1)\subset \ell_2$ translates into $f(\ell_1)\subset \ell_2\cap W^{\perp}$ which is equivalent to $\ell_1\subset \ell_2\cap W^{\perp}+W$. The space $\ell_2\cap W^{\perp}+W$ is strictly contained in $\ell_2+W$ in general.  Note though that the constrains in (\ref{cons1}) do say that in fact the elements of the flag $\ell_*$ are subspaces of complements of $W_*$. Understanding the connections between these objects seems like an interesting question. 
\end{rem}

We do not have a simple recipe  to desingularize a general $W-$flag variety. However, in the case at hand, we got lucky. Let
\[\widehat{G}:=\prod_{1\leq j\leq i\leq k}\Gr_j(V^i_j)
\]
where, for $j\leq i$ define
\[ V^i_j:=\sum_{p=1}^jL_p+\sum_{p=i+1}^{k+1}L_p^{\perp}.
\]
Use the notation $\ell_{*}^*:=(\ell_j^i)_{i,j}\in\widehat{G}$ with $\ell_j^i\in\Gr_j(V^i_j)$.
\begin{theorem} The set
\[ \mwG:=\{\ell_{*}^*\in \widehat{G}~|~\ell_j^i\subset\ell_{j+1}^i,~ \ell_j^i\subset \ell_j^{i+1}+L_{i+1}^{\perp},~\forall~i,j\leq k-1\}
\]
is a complex manifold and the projection $\pi:\mwG\ra \mathcal{G}$:
\[ \ell_*^*\ra (\ell_i^i)_{1\leq i\leq k}
\]
is a direct resolution of the  $\mathcal{G}$.
\end{theorem}
 The resolution $\widehat{\mathcal{G}}$ comes with a well-defined map that extends (\ref{Hoeq1})
 \[ \widehat{\Psi}:\mwG\ra\Fl_{\beta_1,\ldots,\beta_k}(E),\qquad \widehat{\Psi}(\ell_*^*)=\left(\ldots, \ell_i^i+\sum_{j=1}^iL_j^{\perp},\ldots\right)
 \]
 In order to obtain the embedded resolution of $\overline{V_{\beta}}$ inside $\Gr_k(E)$ we recall that $\Fl_{\beta_1,\ldots,\beta_k}(E)$ is the base space of a fiber bundle $\mathcal{F}(E)$ which is a subbundle of the trivial  bundle  $\Fl_{1,\ldots,k}(E)\times \Fl_{\beta_1,\ldots,\beta_k}(E)$ and whose fiber over $F_{\beta}$ is the well-known Kempf-Laksov resolution of $\overline{V_{\beta}}(F)$, namely:
 \begin{equation}\label{KL1}\widehat{V_{\beta}}(F):=\{(L_1,\ldots,L_k)\in \prod_{i=1}^k\Gr_i(F_{\beta_i})~|~L_1\subset L_2\subset \ldots\subset L_k\}\ra \Gr_k(E),\quad L_*\ra L_k.
 \end{equation}
 \begin{theorem} The fiber product $\widehat{\Psi}^*\mathcal{F}(E)$ of $\widehat{\Psi}$ with the projection $\mathcal{F}(E)\ra \Fl_{\beta_1,\ldots,\beta_k}(E)$ together with the natural map:
 \[ \widehat{\Psi}^*\mathcal{F}(E)\ra \Gr_k(E),\qquad (\ell_{*}^*, L_*)\ra L_k 
 \]
 forms an embedded resolution of $\overline{V_{\beta}}$ inside the Grassmannian $\Gr_k(E)$.
 \end{theorem}
 
The question of whether this embedded resolution is isomorphic with some Bott-Samelson type construction seems to us quite interesting.

The results of this note aimed at presenting alternative ways of constructing desingularizations for some thoroughly studied objects like the Schubert varieties. The unifying thread and the main novelty was the use of  incidence relations in two directions, or bioriented flag manifolds.  While the Bott-Samelson construction \cite{BS,De,Ha} or the algorithm in characteristic $0$ of building resolutions via successive blow-ups as developed by Hironaka \cite{Hi} and later simplified by Bierstone-Milman \cite{BM}, Villamayor\cite{Vi}, Wlodarczyk \cite{Wl} and others achieve the same purpose, the constructions we produce here are quite elementary. In the $C^\infty$ setting one should also consult the results of Duan \cite{Du} and Harvey and Lawson \cite{HL}.\\

\noindent
\emph{Acknowledgements:} Most of this work was done during the period I spent at the  Institutul de Matematica al Academiei Romane (IMAR) in 2017. I would like to thank  the Institute for hospitality, especially  Lucian Beznea for invitation and support and   Cezar Joita and Sergiu Moroianu for interesting mathematical conversations.

This article benefitted a lot from the comments of the anonymous referee to whom I am clearly indebted. It was her/his suggestion that the bioriented flag resolution of $S^w(F)$ was the bubblesort Bott-Samelson resolution, now Theorem \ref{BBS},  proved only after an initial version of this paper was made public.

\section{The resolution of the flag Schubert varieties} 

We fix a complete flag $F_*\in \Fl_{1,2,\ldots n}(E)$ and a complementary decreasing flag $G^*$ such that
\[ F_i\oplus G^i=E.
\]
The Schubert cell in $\Fl_{1,2,\ldots n}(E)$ associated to $w\in S_n$ (and $F_*$) is
\begin{equation}\label{Sw1} S^w:=\{\ell_*\in\Fl_{1,2,\ldots n}(E)~|~\dim{\ell_p\cap F_q}=\#\{i\in\{1,\ldots, n\}~|~i\leq p,~ w(i)\leq q\}
\end{equation}
Another way to put it is that if  $\ell_*\in S^w$, then for all all $i$, $w(i)$ is the  unique place of jump  (discontinuity) for the non-decreasing function
\begin{equation}\label{eq4} \{0,\ldots,n\}\ni k\ra \dim{\ell_i\cap F_k}-\dim{\ell_{i-1}\cap F_k} \in \{0,1\}.
\end{equation}
where by definition $\ell_0=F_0=\{0\}$.
The Schubert variety is the closure of the Schubert cell $S^w$ and is described by inequalities $\geq$  in place of $=$ in (\ref{Sw1}).

Notice that the flag $F^w_*\in S^w$ where $F^w_p:=\displaystyle{\bigoplus_{j=1}^{p}} F_{w(j)}\cap G^{w(j)-1}$  since
\[ \dim{F^w_p\cap F_q}=\# w(\{1,\ldots, p\})\cap\{1,\ldots, q\}=\#\{i~|~i\leq p, w(i)\leq q\}.
\]
\begin{definition} The rank matrix associated to $S^w$ is the $n\times n$ matrix of integers  $D^w$ with
\[ d_{pq}:=D_{pq}^w=\#\{i\in\{1,\ldots, n\}~|~i\leq p,~ w(i)\leq q\}
\]
\end{definition}
The rank matrix has the property that it is \emph{slowly} increasing along every line and every column, i.e. $d_{p,q+1}-d_{p,q}\in\{0,1\} \ni d_{p+1,q}-d_{p,q}$ and that $d_{pn}=p=d_{np}$.  In fact, the difference between two consecutive lines is represented by the  values of the  function $(\ref{eq4})$ for $k>0$.

\begin{definition} \label{Def22}  The product  space $\mathcal{P}^w$ associated to $D^w$ is the product of $n\times n$ spaces where the space on position $(p,q)$ is $\Gr_{D_{pq}^w}:=\Gr_{D_{pq}^w}(E)$. 
\end{definition}
Hence the first line contains only point spaces and projective spaces, the second line contains point-spaces, projective spaces and Grassmannians of $2$-planes, etc.

 \begin{definition} The bioriented flag  associated to the dimension matrix $D^w$ is the subset $\Fl^{w}\subset \mathcal{P}^w$ which consists of linear subspaces $\ell_{p,q}\in\Gr_{D_{pq}^w}$ such that 
 \[ \ell_{p,q}\subset \ell_{p,q+1}\quad\mbox{and}\quad\ell_{p,q}\subset\ell_{p+1,q}
 \]
 Denote by $\ell_{*,*}\in \Fl^w$ a general element.
 \end{definition}
 \begin{rem}\label{Redrem} There is a lot of redundancy in the definition of the bioriented flag, in the sense that if $\ell_{*,*}\in\Fl^{w}$ then some of its coordinates in a line, or in a column might be equal. One can safely remove this redundancy by eliminating already from the product space $\mathcal{P}^w$  the copies of the same Grassmanian which appear either along a line or along a column and keeping just the first copy, where by first we mean the first from left to right and from up to down. However, eliminating redundancy brings quite of bit of complication, and not only notational. The simplicity of the construction is directly related to the presence of the superfluous copies of the Grassmannian. 
 
  Since the essential set of Fulton \cite{Fu1} determines the rank function and therefore the rank matrix for every $w$ it will of course determine the number of copies of $\Gr_i$ obtained after eliminating redundancy. While it certainly seems interesting to have such a direct relation,  we take a different path here, namely we devise a simple algorithm that starts with the graph of the permutation and delivers the number copies of $\Gr_i$ for each $1\leq i\leq \dim {E}-1$ that appear in this process after eliminating redundancy.  Rather than writing the entire rank matrix and then deleting the superfluous entries, one can proceed as follows.
  \begin{itemize}
  \item[(i)] Write down the  directed graph which is obtained from  the natural partial order relation induced by the lexicographic order on the graph\footnote{The other notion of  graph, i.e. pairs $(i,w(i))$ called here labels.} of $w$, i.e. $(i,w(i))<(j,w(j))$ iff  $i<j$ and $w(i)<w(j)$.  
  \item[(ii)] Organize the directed graph as a "building" with the smallest floor at level $1$, containing the set of minimal labels with respect to the order relation of item (i). Suppose that indegree\footnote{The number of arrows   connecting it to its smaller, immediate neighbours in the graph} of $(i,w(i))$ is $k-1$. Put then $(i,w(i))$  at level $k$. 
  \item[(iii)] For the apartments (read labels) on a fixed floor there exists a total order relation  given by comparing the first coordinates of the labels. 
  \item[(iv)] Open new "apartments" in the building, on floors $2\leq i\leq n$ by doing the following. For each pair $(p,w(p)), (q,w(q))$ of \emph{consecutive} (with respect to the  order of item (iii)) labels  on floor $i-1$  open an apartment at level $i$ with the new label given by the pair $(\max{\{p,q\}},\max{\{w(p),w(q)\}})$.  
    \item[(v)] There exists a unique apartment on level $n$ with label $(n,n)$ which will be  discarded. 
    \item[(vi)] Replace each label/apartment at level $1\leq i\leq n-1$ by a copy of $\Gr_i$.
          \end{itemize}
          
          Item (v) can be justified by noticing that the label of an apartments on a given floor $i$ has both components bigger or equal $i$.    
          
          This algorithm puts together a few easy to justify facts. We will use the expression "non-redundant" as a synonym for survival at the end of the elimination process. 
          \begin{itemize}
          \item[(1)] For any $i$ and for any non-redundant copy of $\Gr_i$ there exists no other non-redundant copy of $\Gr_i$ on the same row or column. Moreover all the  non-redundant copies of $\Gr_i$ live either in the north-east region (rectangle) or in the south-west region (rectangle) of any other non-redundant  $\Gr_i$.
          \item[(2)] For any $i$, the non-redundant immediate neighbours to the north and to the west of any non-redundant $\Gr_i$ are either copies of $\Gr_{i-1}$ or the boundary of the matrix while the immediate neighbours to the south and west are copies of $\Gr_{i+1}$ or the boundary of the matrix.
            \item[(3)] For a fixed $i$ there exists a total order on the non-redundant copies of $\Gr_i$: say that $\Gr_i<\Gr_i'$ if $\Gr_i'$ appears in the north-east region of $\Gr_i$ or equivalently $\Gr_i$ appears in the south-west region of $\Gr_i'$. 
            \item[(4)] For any $i$ and two consecutive (with respect to (3)) non-redundant copies of $\Gr_i$ there exists a unique common immediate neighbour (necessarily a $\Gr_{i+1}$) to the east of the first and to the south of the second. The coordinates of the position of the common neighbour are expressed by the maximum of the coordinates of the two copies of $\Gr_i$.
            \item[(5)] The only redundancy in the graph $\{(i,w(i))~|~1\leq i\leq n\}$ appears when $w(n)=n$, otherwise the copies of the Grassmannian  on these positions survive the elimination process.
          \end{itemize}
          
          An important observation which  is obvious in the matrix product of spaces is that every floor has at least one apartment since we have at least one copy of $\Gr_i$ for all $1\leq i \leq n$.
          
 \end{rem}

 We will make use of the building description of Remark \ref{Redrem} in order to prove the following.
 \begin{prop}\label{P01} After eliminating redundancy there are a total of $l(w)+n-1$ copies of Grassmannians $\Gr_i$ with $1\leq i\leq n-1$ left in $\Fl^w$.
 \end{prop}
 \begin{proof} First we prove a  dichotomy for which it seems easier to look at the matrix product of spaces. The dichotomy is the following: on the floor $n-1$ there is exactly one apartment if and only if $w(n)=n$, otherwise there are exactly two apartments. Recall that the apartments on floor $n-1$ correspond to the copies of $\Gr_{n-1}$. The copies of $\Gr_{n-1}$ can only appear on positions $(n-1,n-1)$, $(n,n-1)$ and $(n-1,n)$. But they do definitely appear on positions $(n,n-1)$ and $(n-1,n)$ since the row and column $n$ of the matrix of spaces are the same irrespective of $w$. If $\Gr_{n-1}$ appears on position $(n-1,n-1)$ then it renders the other two copies redundant. If it does not appear then the other two copies survive the elimination process. Since the last row is always 
 \[ \xymatrix{ \bP^1 & \Gr_2& \ldots & \Gr_{n-1} & \Gr_n}\]
we conclude that $\Gr_{n-1}$ appears on position $(n-1,n-1)$ if and only if the unique change between the last two rows appears in the last column. In other words if and only if $w(n)=n$.
 
 We use now induction on $n$ to prove our formula. First notice that $(n,w(n))$ always appears on floor $w(n)$. Moreover this label is the last label on the floor $w(n)$. Let $w':S_{n-1}\ra S_n\setminus \{w(n)\}$.
 
 If $w(n)=n$, then we already know from the first claim that there exists a unique apartment at level $n-1$. Moreover $l(w)=l(w')$ and $w'$ determines a building that ends at level $n-2$. Hence there exists only one extra apartment, apart from the ones determined by $w'$. By induction this gives a total of $l(w')+(n-2)+1=l(w)+n-1$ apartments.
 
 If $w(n)\neq n$ then we notice that $w'$ will still determine an expanded building  by using exactly the same algorithm as before as if $w'$ would belong to $S_{n-1}$ (see also Remark \ref{Bij}).  However, we keep the labels of $w'$. Now, either $w(n)=n-1$ (top floor) or the apartment $(n,w(n))$  will give rise to only one apartment on floor $w(n)+1$ which will have also have the last label of this floor as the first coordinate will be $n$.  This will happen via step iv of the algorithm, using the last apartment of the same floor $w(n)$ of the building determined by $w'$ (this exists!). And the process of creation of new apartments continues always creating an apartment with the last label at level $w(n)+2$, etc. Either way, $(n,w(n))$ will  gives rise to $(n-1)-(w(n)-1)=n-w(n)$  new apartments till level $n-1$, including itself. 
 
 There is one more apartment to be counted. Recall from the first paragraph that if $w(n)\neq n$ there are two apartments at level $n-1$. We have already counted $(n,n-1)$ because it was created by $(n,w(n))$ along the way. But we also have $(n-1,n)$ which did not appear in the building generated by $w'$ because there we stopped at level $n-2$. 
 
 Therefore by induction, the total number of apartments is
 \[ [l(w')+(n-2)]+(n-w(n))+1=l(w)+n-1.
 \]
   \end{proof}
   
    \begin{rem} \label{Bij} The set of bijections $\mathcal{B}_n^w$
 \[ \{1,\ldots,n-1\}\ra \{1,\ldots,n\}\setminus \{w(n)\} \]
 is a torsor for the group $S_{n-1}$, i.e. $S_{n-1}$ acts freely on the right and transitively on the set of these bijections.  From many points of view these bijections behave as permutations in $S_{n-1}$.  
  
  This is related to the fact that this torsor has a preferred point $\circ$ which is the unique increasing bijection. One can identify $\mathcal{B}_n^w$ with $S_{n-1}$ via the action on $\circ$. Via such an identification, any bijection in $\mathcal{B}_n^w$ gets a rank $(n-1)\times (n-1)$ matrix and correspondingly a matrix product of spaces. For example, for the restriction $w$ of $w$ to $\{1,\ldots,n-1\}$ the matrix product of spaces $\mathcal{P}^{w}$ is the same as the one obtained by erasing the last row and column $w(n)$ in $\mathcal{P}^w$.  \end{rem} 
  
   \begin{example} Take $\sigma:=(4\;8\;6\;2\;7\;3\;1\;5)$ (the same choice appears in \cite{Fu}). The directed graph of $\sigma$ written in "building" format with increasing entries on each floor  is the following:
 
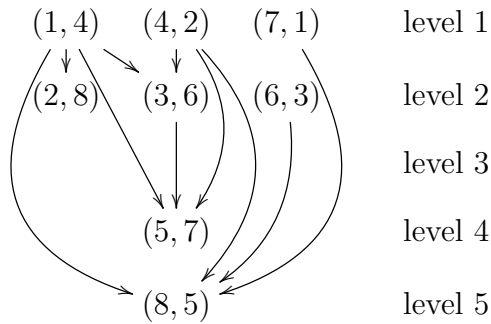
\[ \xymatrix@=0.3cm{ (1,4) \ar[d]\ar@/_{3.5pc}/[ddddr] \ar[dr] \ar[dddr]& (4,2)\ar[d]\ar@/^{1.5pc}/[ddd]\ar@/^{2.5pc}/[dddd] & (7,1)\ar@/^{3.5pc}/[ddddl] && \mbox{level 1}\\
   (2,8)& (3,6)\ar[dd] & (6,3)\ar@/^{1.5pc}/[dddl] && \mbox{level 2}\\
    &&&&\mbox{level 3}\\
    &(5,7)& && \mbox{level 4}\\
   & (8,5)& && \mbox{level 5}}
 \]
 The "expanded building with all open apartments is:
 \[\xymatrix@=0.01cm{ (1,4)&  & &(4,2)&& &(7,1)  \\
   &(2,8)& (3,6)& (4,4) & &(6,3)  &(7,2) \\
    &&(3,8) &(4,6)& &(6,4)&(7,3)\\
    &&&(4,8)&(5,7)&(6,6) &(7,4) \\
   && &&(5,8)&(6,7) &(7,6)&(8,5)\\
   &&&&&(6,8)&(7,7)&(8,6)\\
   &&&&&&(7,8)&(8,7)
    }
    \]
    
     There are $25$ Grassmannians in total. Since $l(\sigma)=18$, Proposition \ref{P01} is verified.
     \[(\bP^1)^3\times (\Gr_2)^5\times (\Gr_3)^4\times (\Gr_4)^4\times (\Gr_5)^4\times (\Gr_6)^3\times (\Gr_7)^2\]
     
 By comparison, in Definition \ref{Def22} the product of spaces induced by the rank matrix of $\sigma$ is:
 \[(\bP^1)^{18}\times (\Gr_2)^{10}\times (\Gr_3)^8\times (\Gr_4)^6\times (\Gr_5)^4\times (\Gr_6)^3\times (\Gr_7)^2.
 \]
 \end{example}
   
    \vspace{0.1cm}
   \begin{prop}\label{P11} The bioriented flag $\Fl^{w}$ is a compact, smooth complex manifold.
   \end{prop}
   \begin{proof} The proof more generally applies to any bioriented flag which is a subspace of an $n\times m$ \emph{increasing} matrix product  of Grassmannians with left to right and up-down inclusion relations. By an increasing matrix product of Grassmannians we mean that the dimensions of the subspaces increase from left to right and  up down. The proof goes by induction on the number of rows. It is trivially true for $n=1$ since, in this case, the bioriented flag becomes a standard flag manifold. 
   
  For  the case $n=2$ project the bioriented flag onto the product of Grassmannians of the second row. The image is the flag manifold obtained by erasing the first row together with the vertical inclusion conditions. Fix a point $F_*: F_{1}\subset \ldots \subset F_{m}$  where $\dim F_i<\dim F_{i+1}$ in the target. Then the fiber\footnote{A non-empty fiber argument goes as follows. Start with any $\ell_1\in \Gr_{\alpha_1}(F_1)$, then choose any $\ell_2'=\ell_2/\ell_1\in \Gr_{\alpha_2-\alpha_1}(F_2/\ell_1)$ then choose any $\ell_3/\ell_2$ in $\Gr_{\alpha_3-\alpha_2}(F_3/\ell_2)$ etc. } over $F_*$ consists of finite sequences of subspaces
  \[ \ell_1\subset\ldots\subset \ell_m,\qquad \ell_i\subset F_i,\;\; \ell_i\in \Gr_{\alpha_i}.
  \]
  for some fixed $\alpha_i,$ $1\leq i\leq m$. This fiber is a manifold whose biholomorphism type does not depend on the point $F_*$ and the proof of this fact proceeds by induction on the number of columns. Clearly true for $m=1$. In general, project onto the product $\prod_{1\leq i\leq m-1}\Gr_{\alpha_i}(F_i)$. By induction hypothesis the image is a manifold. The fiber consists of $\ell_m\in \Gr_{\alpha_m}(F_m)$ such that $\ell_m\supset \ell_{m-1}$. But this is $\Gr_{\alpha_m-\alpha_{m-1}}(F_m/{\ell_{m-1}})$. Hence when $n=2$ we get a fibration over a flag manifold with the fiber itself being  a tower of fiber bundles with Grassmannian fibers.
     
 For general $n$ consider the projection to the product of spaces of the $(n-1)\times m$ submatrix  induced by the last $n-1$  rows.  The image will be a bioriented flag inside a product of $(n-1)\times m$ spaces obtained by erasing the first row and the corresponding vertical inclusions conditions. By induction this is a manifold. The fiber of this projection is uniquely determined by fixing the second row of spaces. Suppose this second row is $F_{1}\subset \ldots\subset F_m$. Then the fiber consists of subspaces
 \[\ell_1\subset\ldots\subset \ell_m,\qquad \ell_i\subset F_i,\;\; \ell_i\in \Gr_{\alpha_i}
 \]
 where $\alpha_i$ are fixed. Hence the fiber in the general case is no different from the fiber for $n=2$.
 
   Hence one gets that $\Fl^{w}$ is a tower of fiber bundles and in fact since $\alpha_m-\alpha_{m-1}\leq 1$ the fibers are projective spaces.
      \end{proof}
 
We fix now the coordinates on the last line. Let
\[ \hat{S}^w:=\{\ell_{*,*}\in\Fl^{w}~|~\ell_{n,i}=F_{i}\}.
\]
\begin{prop} \label{P12} The space $\hat{S}^w$ is a smooth manifold of dimension $l(w)$, the length of $w$, i.e. the number of inversions of $w$.
\end{prop}
\begin{proof} Just like Proposition \ref{P11}, the proof that $\hat{S}^w$ is a manifold applies more generally to bioriented flags with \emph{fixed last row} in any $n\times m$ increasing matrix product of Grassmannians and proceeds by induction starting with the case $n=2$ already treated in Proposition \ref{P11}. The difference in proof is that the induction step uses the projection onto the line $n-1$. The fiber this time is by induction hypothesis a manifold, while the image is again a manifold as it leads again to the case $n=2$. We do the details for $\hat{S}^w$ in order to justify the claim about $\dim{\hat{S}^w}$.

 It is convenient to think alternatively $\hat{S}^w$ as a subspace of an $(n+1)\times (n+1)$ product of spaces rather than just of $\mathcal{P}$ by introducing a line and a column of point spaces and hence introducing the coordinates $\ell_{0,i}=pt$ and $\ell_{i,0}=pt$, for all $i\in\{0,\ldots n\}$. 

Let $\ell_{*,*}\in \hat{S}^w$ and look at $\ell_{n-1,*}$. This is an element of a product 
\[ \prod_{i=0}^{n-1} \Gr_{a_i}(E)
\]
where $a_i\in\{0,1,\ldots n-1\}$ is an increasing sequence of numbers such that $a_i-a_{i-1}\in \{0,1\}$ with $a_0:=0$. In fact, there exists exactly one $i$, namely $i=w(n)$ such that
\[a_i-a_{i-1}=0
\]
and this is due to the fact that the function in (\ref{eq4}) has exactly one jump point. Now, because of the vertical $\subset$ relations and because the last line of $\ell_{n,*}=F_*\in \hat{S}^w$ is fixed  we get that $\ell_{n-1,j}=F_j$ for all $j\leq w(n)-1$. 

We also have that $\ell_{n-1,w(n)}=\ell_{n-1,w(n)-1}=\ell_{n,w(n)-1}=F_{w(n)-1}$, the first equality holding because $\Gr_{a_{w(n)}}=\Gr_{a_{w(n)-1}}$.  

The only freedom appears in the choice of the coordinates $\ell_{n-1,j}$ for $j\geq w(n)+1$. But, in this case  we can factor out $\ell_{n-1,w(n)}=F_{w(n)-1}$ from all $\ell_{n-1,j}$ and from all $\ell_{n,j}$ for $j\geq w(n)+1$. Then
\begin{equation}\label{eqidn} \frac{\ell_{n-1,w(n)+1}}{\ell_{n-1,w(n)}}\subset\ldots \subset \frac{\ell_{n-1,n}}{\ell_{n-1,w(n)}}
\end{equation}
is a flag of subspaces inside $E/F_{w(n)-1}$ of dimensions $1,2,\ldots n-w(n)$ such that
\begin{equation} \label{eqidn1}  \frac{\ell_{n-1,w(n)+i}}{\ell_{n-1,w(n)}} \subset \frac{F_{w(n)+i}}{F_{w(n)-1}}=\frac{\ell_{n,w(n)+i}}{F_{w(n)-1}},\qquad\forall i\geq 1\end{equation} The   spaces on the right hand side have  dimensions $i+1$, with $i=1,\ldots, n-w(n)$. In other words the set of flags of type (\ref{eqidn}) satisfying (\ref{eqidn1}) represents the Kempf-Laksov resolution (see (\ref{KL1})) of the Grassmannian Schubert variety inside $\Gr_{n-w(n)}(E/F_{w(n)-1})$  defined by the multi-index $\beta_i=i+1,\; 1\leq i\leq n-w(n)$ and fixed flag $\left(\frac{F_{w(n)+i}}{F_{w(n)-1}}\right)_{1\leq i\leq n-w(n)}$. The dimension of such a resolution is the dimension of the Schubert cell, i.e.
\[\sum_{j=1}^{n-w(n)} (\beta_j-j)=n-w(n)
\] 
Now $n-w(n)=\#\{1\leq k\leq n~|~w(k)>w(n)\}$ is the number of inversions induced by $w(n)$. 

Moreover the image of the projection of $\hat{S}^w$ onto the line $n-1$ of $\hat{S}^w$ is a smooth manifold  of dimension $n-w(n)$. 

One proceeds by induction as explained at the beginning. To understand what happens next, fix now the line $n-1$. Notice that we can first eliminate from discussion the column $w(n)$ of spaces   once and for all since the Grassmannians that appear in $\mathcal{P}^w$ in this column above row $n$ are repetitions of the Grassmannians that appear in column $w(n)-1$ hence the horizontal $\subset$ relations imply their redundancy. Therefore one gets a bioriented flag of spaces in a product of $(n-1)\times (n-1)$ spaces with a fixed last line (the line $n-1$ in the original matrix). Project onto the line $n-2$ in order to get another smooth manifold with dimension equal to the number of inversions determined by $w(n-1)$ by regarding now (via restriction)  $w$ as a bijection $\{1,\ldots, n-1\}\ra \{1,\ldots,n\}\setminus \{w(n)\}$. 

One gets a tower of fiber bundles with fibers given by Kempf-Laksov resolutions of Grassmannian Schubert varieties. The total dimension is the number of inversions of the permutation $w$. 
\end{proof}
\begin{rem} Let $m:=n-w(n)$. Then $\dim (E/F_{w(n)-1})=m+1$ and the Schubert variety in $\Gr_{m}(E/F_{w(n)-1})$ associated to the index sequence $\beta_i:=i+1$, $i=1,\ldots, m$ and fixed flag
 \[ F^{\beta_i}:=\frac{F_{w(n)+i}}{F_{w(n)-1}}, \qquad 1\leq i\leq m
 \]
 which appears in the proof is in fact the full Grassmannian $\Gr_{m}(E/F_{w(n)-1})$ since all conditions
 \[\dim L\cap F^{\beta_i}=i
 \]
 are open conditions. 

 The Kempf-Laksov resolutions of the Grassmannian of hyperplanes $\Gr_{m}(E/F_{w(n)-1})$ seen as a Schubert  variety as described in the previous paragraph is a tower of fiber bundles with $\bP^1$-fiber.

We conclude  that $\hat{S}^w$ is a tower of fiber bundles with $\bP^1$ fibers.
\end{rem}
\begin{rem} There exists some extra redundancy in the definition of $\hat{S}^w$, other than the one already present in  $\Fl^w$. First, by fixing the last row we have that $\hat{S}^w$ is a subset of the $(n-1)\times n$ submatrix product of Grassmannian spaces which are used in the definition of $\Fl^w$. Second, since the first $w(n)$ components of the row $n-1$ are also fixed one can safely eliminate the Grassmannians they belong to. 

In  the "building"  description given in Remark \ref{Redrem}, we therefore need to "close" every apartment with the last label on each floor.  By Proposition \ref{P01}, we are thus left, with a total of $l(w)$ Grassmannian spaces. This is the same number of spaces that appears in the description of the Bott-Samelson resolution as the closure of an orbit as constructed by Magyar \cite{Ma1}, (see also \cite{JW}). It is no coincidence. We will see in Theorem  \ref{BBS} that $\hat{S}^w$ gives a resolution that coincides with the "bubblesort" Bott-Samelson resolution.

For  $\sigma:=(4\;8\;6\;2\;7\;3\;1\;5)$ the  spaces that remain after eliminating redundancy are obtained from the next diagram and written in the last column:
\[\xymatrix@=0.02cm{ (1,4)&  & &(4,2)&&  &&&& (\bP^1)^2\times \\
   &(2,8)& (3,6)& (4,4) & &(6,3)  & & &&(\Gr_2)^4\times\\
    &&(3,8) &(4,6)& &(6,4)& &&& (\Gr_3)^3\times\\
    &&&(4,8)&(5,7)&(6,6) & &&& (\Gr_4)^3\times\\
   && &&(5,8)&(6,7) &(7,6) & &&(\Gr_5)^3\times\\
   &&&&&(6,8)&(7,7) &&& (\Gr_6)^2\times\\
   &&&&&&(7,8)&&& \Gr_7.
    }
\]
\end{rem}
\begin{theorem}\label{Flres1} The projection $\pi:\hat{S}^w\ra \prod_{i=1}^n \Gr_i$ 
\[  \ell_{*,*}\ra \ell_{*,n}
\] 
is a resolution of $S^w\subset \Fl_n$.
\end{theorem}
\begin{proof} Let  $\ell_p:=\ell_{p,n}$ where $\ell_{*,*}\in \hat{S}^w$. Then due to the incidence relations both $\ell_p$ and $F_q$ will contain $\ell_{p,q}$ for all $q$. The dimension of $\ell_{p,q}$ is the number of rows $i\leq p$ such that $w(i)\leq q$. Hence
\begin{equation}\label{eq3456}\dim{\ell_p\cap F_q}\geq \dim \ell_{p,q}= D^w_{pq},\qquad \forall p,q.
\end{equation}
with equality if and only if $\ell_p\cap F_q=\ell_{p,q}$. 

In other words, the projection lands within the Schubert variety $\overline{S^w}$. The image clearly contains $S^w$ by taking $\ell_{p,q}:=\ell_p\cap F_q$ where $\ell_{*}\in S^w$ is given. By the characterization of the equality case in (\ref{eq3456}) the projection $\pi$ restricted to $\pi^{-1}(S^w)$ is a biholomorphism onto $S^w$.

Moreover since the projection map goes between smooth algebraic varieties and the projection map is algebraic and $\hat{S}^w$ is compact,  it means that the whole $\overline{S^w}$ will be in the image. 
\end{proof}

We will show next  that the resolution $\hat{S}^w$ is isomorphic to a particular Bott-Samelson resolution of $S^w$. We use Magyar's construction of Bott-Samelson resolutions as presented in Section 5 of \cite{JW}. For every presentation 
\[ w=w_1\ldots w_{l(w)}\]
 as a reduced word, i.e. as a minimal product of adjacent transpositions\footnote{These are $s_i\in S_n$ swapping $i$ and $i+1$.} one can construct a Bott-Samelson resolution as follows:
 \begin{itemize}
 \item[(i)] identifying $w_j$ with $s_{d_j}$ for some $1\leq d_j\leq n-1$;
 \item[(ii)] considering the product of Grassmannians $G:=\displaystyle\prod_{i=1}^{l(w)}\Gr_{d_i}$ in the ambient space; denote by $[V_1,\ldots ,V_{l(w)}]$ an element of this space;
 \item[(iii)] imposing $l(w)$ incidence relations of type:
 \begin{equation}\label{increl2}     V_{l(j)}\subset V_j\subset V_{r(j)},\qquad 1\leq j\leq l(w)
 \end{equation}
 where ${l(j)}$ and ${r(j)}$ are either indices between $1$ and $l(w)$ defined at item (iv) or are undefined in which case $V_{l(j)}:=F_{d_j-1}$ and $V_{r(j)}:=F_{d_j+1}$ respectively.
 \item[(iv)] $l(j)$ is the greatest index such that $l(j)<j$ and $w_{l(j)}=s_{d_j-1}$ or if no such index exists then $l(j)$ is undefined; similarly,
 $r(j)$ is the greatest index such that $l(j)<j$ and $w_{l(j)}=s_{d_j+1}$ or if no such index exists then $r(j)$ is undefined.
 \end{itemize}

 The presentation of $w$ as a reduced word we will use is the one called the "bubblesort" which we describe next. Notice that the effect of multiplying any $\sigma\in S_n$ to the right with a transposition $s_i$ gives a permutation where $i\ra\sigma(i+1)$ and $i+1\ra \sigma(i)$ while the rest is unchanged. Fix $w\in S_n$. If $w(n)\neq n$, then the composition of transpositions
 \begin{equation}\label{eqh1} s_{w(n)}s_{w(n)+1}\ldots s_{n-1}=e\cdot s_{w(n)}s_{w(n)+1}\ldots s_{n-1}
 \end{equation}
 seen as acting to the right on the identity $e$ has the effect of moving $w(n)$ to position $n$. If $w(n)=n$ then nothing happens and we are effectively reduced to the case $w\in S_{n-1}$.
 
 One can then use the adjacent transpositions of $S_{n-1}$ (seen as transpositions in $S_n$) in order to move $w(n-1)$ to position $n-1$, always multiplying on the right.  Hence one would multiply (\ref{eqh1})  with
 \[ s_{w(n-1)}\ldots s_{n-2} \quad\mbox{or}\quad s_{w(n-1)-1}\ldots s_{n-2}
 \]
 depending on whether $w(n-1)$ stayed fixed after the first step (\ref{eqh1}), i.e. $w(n-1)<w(n)$ or moved one position to the left, i.e. $w(n-1)>w(n)$. The role of the identity permutation $e$ for the initial step is now taken by the restriction of $w$ as a bijection $\{1,\ldots,n-1\}\ra \{1,\ldots,n\}\setminus \{w(n)\}$ (see also Remark \ref{Bij}).

 Continuing in this manner one gets to write $w$ as the bubblesort reduced word of $l(w)$-transpositions. Denote by $\BBS^w$ the (bubblesort) Bott-Samelson resolution induced by this presentation of $w$. The projection map between $\BBS^w$ and $\Fl_{1,2,\ldots,n}(E)$ that gives the resolution of $S^w$ is defined by (compare \cite{JW}):
 \begin{equation}\label{resB}[V_1,\ldots,V_{l(w)}]\ra ( V_{p(1)},\ldots,V_{p(i)},\ldots,E)
 \end{equation}
 where  $p(i)$ for $i\leq n-1$ is the \emph{total  number} of transpositions needed to bring $w(n)$ to the last spot, $\ldots$,$w(i+1)$ to spot $i+1$. In other words:
 \[p(i)=\sum_{j=i+1}^n\#\{1\leq k\leq n~|~w(k)>w(j)\}
 \]
 and this is also the index of last occurence of $s_i$ in the bubblesort reduced word for $w$.

We now write 
\[ w:=t_1\ldots t_{n-1}
\]
where $t_1:=s_{w(n)}s_{w(n)+1}\ldots s_{n-1}$ is a product of transpositions in $S_n$ that brings $w(n)$ to position $n$, $t_2$ is a product of transpositions in $S_{n-1}$ acting on $t_1$ to the right that brings $w(n-1)$ on position $n-1$, etc. Let
\[ G^{t_1}:=\Gr_{w(n)}\times \ldots\times \Gr_{n-1},
\]

In what follows next we will assume that $w(n)\neq n$ so that $G^{t_1}$ is not trivial.

Notice that $G^{t_1}$ is the subproduct of the first $n-w(n)$ spaces of $G$, the product of Grassmannians where  $\BBS^w$ lives. Define $G^{t_2}$,\ldots $G^{t_{n-1}}$ in an analogous manner such that $G=G^{t_1}\times \ldots\times G^{t_{n-1}}$.  There exists a (restriction of the) projection onto the first $n-w(n)$ coordinates:
\[\pi_1:\BBS^w\ra G^{t_1}.
\]

The product $w':=t_2\ldots t_{n-1}$ is a permutation in $S_{n-1}$ and is already given as a product of transpositions by replacing $t_2,\ldots t_{n-1}$.  Therefore there exists a Schubert variety $S^{w'}$ and a corresponding $\BBS^{w'}$ resolution. 

\begin{prop}\label{P234} \begin{itemize}
\item[(i)] Let $m:=n-w(n)$. The image of $\pi_1$ is biholomorphic with the Kempf-Laksov resolution of the Grassmannian $\Gr_{m}(\bC^{m+1})$ seen as a Grassmannian Schubert variety for the multi-index  $\beta_i=i+1$, $i=1,\ldots,m$ and fixed flag $\bC^{\beta_*}$.
\item[(ii)] Every fiber of $\pi_1$ is biholomorphic with $\BBS^{w'}$.
\end{itemize}
\end{prop}
\begin{proof} The proof of (i) consists in simply verifying that the incidence relations (\ref{increl2}) for $1\leq j\leq n-w(n)$,  are equivalent with the incidence relations of the Kempf-Laksov resolution. This is straightforward.

For (ii) let $[W_1,\ldots, W_{n-w(n)}]\in \Gr_{w(n)}\times \ldots\times \Gr_{n-1}$ be a point in the image of $\pi_1$. Replace the original flag $F_*$ by the flag:
\[ F'_*:\qquad F_1\subset \ldots \subset F_{w(n)-1}\subset W_1\subset W_2\subset\ldots \subset W_{n-w(n)}
\]
Let $j\leq l(w)$ such that $j>n-w(n)$. Then there are three mutually exclusive possibilities
\begin{itemize}
\item[(i)] $l(j)$ is defined and  $l(j)> n-w(n)$; in this case the inclusion relation $V_{l(j)}\subset V_j$  is a relation in $G^{t_2}\times G^{t_3}\times \ldots \times G^{t_{n-1}}$ and there exist one  identical relation in $\BBS^{w'}$ for the flag $F'_*$;
\item[(ii)] $l(j)$ is defined and $l(j)\leq n-w(n)$; in this case the inclusion relation $V_{l(j)}\subset V_j$ becomes a relation of type $W_{l(j)}\subset V_j$; notice also that in this situation $l(j)< n-w(n)-1$ since the corresponding $d_j-1$ is strictly smaller than $n-2$ and the  $(n-w(n)-1)$-th index corresponds to $\Gr_{n-2}$; these are relations that involve $F'_{*}$; there exists an identical relation in $\BBS^{w'}$ for the flag $F'_*$;
\item[(iii)] $l(j)$ is undefined; in this case the relation $V_{l(j)}\subset V_j$ is  $F_{d_j-1}\subset V_j$ but then necessarily $d_j\leq w(n)$ since if $d_j>w(n)$ there will  exist an index $k$ in $\{w(n),\ldots, n-1\}$ such that $k=d_j-1$ which would say that $l(j)$ is actually defined; there exists an identical relation in $\BBS^{w'}$ for the flag $F'_*$.
\end{itemize}

Similarly, there are three mutually exclusive possibilities concerning the inclusion relations $ V_j\subset V_{r(j)}$. If $r(j)$ is not defined then $V_j\subset F_{d_j+1}$ and $d_j+1< w(n)$ (otherwise $r(j)$ is defined). Hence these relations involve $F'_{*}$ and there exists an identical relation for $\BBS^{w'}$.

We argue now that for all $j>n-w(n)$, the inclusion relations $V_j\subset V_{r(j)}$ collectively imply that $V_j\subset W_{n-w(n)}$ for all $j$. We already saw this for undefined $r(j)$. If $r(j)$ is defined but $r(j)\leq n-w(n)$ this is almost tautological since then $V_j\subset V_{r(j)}\subset W_{n-w(n)}$. Finally, if $r(j)>n-w(n)$ we proceed recurrently since $r(j)<j$ by definition hence we look at $r(r(j))$,  which will satisfy the same trichotomy. The process will have to produce an $r^{(\alpha)}(j)$ which is either undefined or $\leq n-w(n)$. We have $V_{j}\subset V_{r(j)}\subset V_{r(r(j))}\subset \ldots \subset W_{n-w(n)}$ etc.

 Hence all spaces are within $W_{n-w(n)}$.

We conclude that all inclusions $V_{l(j)}\subset V_j\subset V_{r(j)}$ for $j>n-w(n)$ for the original flag are equivalent with inclusions that involve the flag $F'_*$ which define a $\BBS^{w'}$ resolution.

\end{proof}

\begin{theorem} \label{BBS} There exists an isomorphism of resolutions $\hat{S}^w\ra \BBS^w$, i.e. a biholomorphism that commutes with the projections to $S^w$.
\end{theorem}
\begin{proof} Recall   $\mathcal{P}^w$ of Def. \ref{Def22}. Denote by $(\mathcal{P}')^w$ the product of spaces in $\mathcal{P}^w$ where we erased the last row. Let $(\mathcal{P}')^w_i$ be the product of spaces on row $1\leq i\leq n-1$ of $(\mathcal{P}')^w$. We define a projection map:
\[ \pi:(\mathcal{P}')^w\ra G.
\]
as a product of projections $\pi_i^w:(\mathcal{P}')^w_i\ra G^{t_i}$ where $\pi_i^w$ is defined inductively by starting with row $n-1$. Define $\pi_{n-1}^w$ by projecting onto the components $w(n)+1,\ldots,n$. These correspond to the Grassmannians $\Gr_{w(n)}\times \ldots\times \Gr_{n-1}$. 

For all subsequent $\pi_i^w$ the coordinate corresponding to column $w(n)$ is killed. In other words we need define $\pi_{i}^w$ on the product of spaces in $(\mathcal{P}')^w_i$ in which the space on position $(i,w(n))$ is omitted. Define then $\pi_{n-2}^w$ as $\pi_{n-2}^{w'}$ where $w':\{1,\ldots,n-1\}\ra \{1,\ldots,n\}\setminus \{w(n)\}$ is the restriction of $w$. We use here Remark \ref{Bij} in order to associate to $w'$ a matrix product of spaces  $\mathcal{P}^{w'}$ and corresponding $(\mathcal{P}')^{w'}$. It is easy to see that $\pi_{n-2}^{w'}$ is well-defined, i.e. that it lands in $G^{t_2}$.

We first notice that $\pi$ defined in this way makes a commutative diagram with the projections from $(\mathcal{P}')^w$ to the product of spaces on the last column and the projection from $G$ to $\prod_{i=1}^{n-1}\Gr_i$ described at (\ref{resB}).

Finally, this is a biholomorphism, since by comparing the proof of Theorem \ref{Flres1} with Proposition \ref{P234} both sides are towers of fiber bundles with base-spaces and fibers given by Kempf-Laksov resolutions of the same type of Grassmannians and the projections in the towers commute with $\pi$.
\end{proof}

\section{The Schubert cells in the Grassmannian}

It has been long known that an open dense subset of the regular set of a Schubert variety can be realized as the stable  manifold of a certain critical manifold of a Morse-Bott function on the Grassmannian \cite{Wo}. The critical variety is just a product of projective spaces and the stable manifold fibers as a vector bundle over the critical manifold.  We begin our investigation by describing concretely what this vector bundle is. The particular case of a Schubert variety induced by one incidence relation has already been treated in \cite{HL}. We do not pursue the dynamical point of view here, as we do not need it,  although we emphasize that the results of this section are in the smooth category. The "philosophical" reason why this does not work in the holomorphic category is the general lack of holomorphic sections for surjective holomorphic morphisms between vector bundles. For example, for the quotient universal bundle over a Grassmannian one does not have an embedding into the trivial vector bundle.

Let $\Gr_k(E)$ be the Grassmannian of $k$-linear subspaces of a vector space E of dimension $n \geq  k$. We fix a complete increasing flag
\[F_*~: \quad {0}=F_0 \subset F_1 \subset\ldots\subset F_n =E\]
with $\dim {F_i} = i$. We will index Schubert varieties by strictly increasing sequences of numbers
\[\beta~:\quad 1\leq \beta_1 < \beta_2 <\ldots <\beta_k \leq n\]
\[ \overline{V_{\beta}} :=\{L\in\Gr_k(E)~|~ \dim{L\cap F_{\beta_i}} \geq i\}\]
We will reserve the notation
\[V_{\beta} :=\{L\in \Gr_k(E)~|~ \dim{L\cap F_{\beta_i}} =i\}\]
for an open dense subset of $\overline{V_{\beta}}$ which is smooth. To distinguish between $\overline{V_{\beta}}$ and $V_{\beta}$ we will call the latter the regular Schubert variety as it is a subset of the set of regular points of $\overline{V_{\beta}}$.

We also have the Schubert cell:
\[V_{\beta}^{\circ} :=\{L\in \Gr_k(E)~|~ \dim{L\cap F_j} =i, i=0,\ldots k, \beta_i \leq j<\beta_i+1, \beta_0 :=0, \beta_{k+1} :=\infty\}\]
 Notice that $V_{\beta}$, just like $\overline{V_{\beta}}$ depends only on the choice of nodes $F_{\beta_1} , \ldots , F_{\beta_k}$ in the complete
flag, while the cell $V_{\beta}^{\circ}$ depends also on the nodes $\beta_1-1,\beta_2-1,\ldots ,\beta_k-1$.
The latter is isomorphic with the affine space $\bC^{N_{\beta}}$ where $N_{\beta}:=\sum_{i=1}^k (\beta_i-i)$.
\begin{rem}  The case $k=1$ is uninteresting as $\Gr_1(E) = \mathbb{P}(E)$ and the Schubert variety
$\overline{V_\beta}$ is then $\mathbb{P}(F_{\beta_1})$ which is non-singular. So we might as well assume that $k\geq 2$. 
\end{rem}

We will also fix a complementary, decreasing flag $G^*$ such that
\[ E=F_i\oplus G^i, \qquad \forall i=0,\ldots,n\]
       For $0\leq j <i\leq n$ we define
      \[ F_i^j :=F_i\cap G^j\]
with $ \dim{F_j^i} = i-j$. 

Let $\mathbb{P} :=\prod_{i=1}^k\mathbb{P}\left(F_{\beta_i}^{\beta_{i-1}}\right)$. It has dimension $\beta_k-k$. It comes with an embedding
\[\mathbb{P}\ra\Gr_k(E),\qquad (L_1,\ldots,L_k) \ra \sum_{i=1}^kL_i.
\]
where the sum here is like on many other occasions a direct sum. We have a natural map:
\[ V_{\beta}\ra \mathbb{P},\quad L\ra \left(L\cap F_{\beta_1},P_{F_{\beta_2}^{\beta_1}}(L\cap F_{\beta_2}),...,P_{F_{\beta_k}^{\beta_{k-1}}}(L\cap F_{\beta_k})\right) \]
where the projection $P_{F_{\beta_i}^{\beta_{i-1}}}$ onto $F_{\beta_i}^{\beta_{i-1}}$ is taken with respect to the decomposition 
\[F_{\beta_i}=F_{\beta_{i-1}}\oplus F_{\beta_i}^{\beta_{i-1}}.\]
For every $i=1,\ldots k$, let $\tau_i\ra \mathbb{P}(F_{\beta_i}^{\beta_{i-1}})$ be the tautological line bundle, with the convention that if $\beta_i-\beta_{i-1}=1$ then $\tau_i:=F_{\beta_i}^{\beta_{i-1}}\ra \pt=\bP(F_{\beta_i}^{\beta_{i-1}})$.

Denote by $\tau^{\perp}$ a complement of $\tau$ inside the trivial vector bundle with fiber ${F_{\beta_i}^{\beta_{i-1}}}  $ over $\mathbb{P}\left(F_{\beta_i}^{\beta_{i-1}}\right)$. In  other words:
\[  \tau_i \oplus \tau_i^{\perp} = F_{\beta_i}^{\beta_{i-1}} \]
using the same notation for the vector space and the vector bundle.

It is self-understood that $\tau_i^{\perp}=0$ if $\beta_i-\beta_{i-1} = 1$.
\begin{rem} There is apriori no canonical choice for $\tau_i^{\perp}$ but in the presence of a hermitian metric on $E
$ one can take the orthogonal complement of $\tau_i$. It is due to a lack of holomorphic structure on $\tau_i^{\perp}$ that the results in this section involving $\tau_i^{\perp}$ hold only in the smooth category.
\end{rem}
Define the following vector bundle over $\bP$:
\[\bH:=\Hom(\tau_2,\tau_1^{\perp})\times \Hom(\tau_3,\tau_1^{\perp}+\tau_2^{\perp})\times \ldots\times \Hom(\tau_k,\tau_1^{\perp}+\ldots+\tau_{k-1}^{\perp})
\]
Just to make sure there is no confusion about notation: $\tau_i$ represents both the line bundle
over $\mathbb{P}\left(F_{\beta_i}^{\beta_{i-1}}\right)$ and its pull-back to $\bP$ via the obvious projection. Moreover $\sum_{j=1}^i \tau_j^{\perp}$
is a subbundle of the trivial bundle $\underline{F_{\beta_i}}$ obtained by taking the internal direct sum of the corresponding bundles $\tau_j^{\perp}$ since surely $\tau_j^{\perp}\cap\tau_{j'}^{\perp}$ = {0} for all $j\neq j'$.

Notice that $\bH$ is a bundle of rank $\dim{V_{\beta}}-\dim{\bP}$.
Let 
\[ \Phi:\bH\ra \Gr_k(E),\qquad \Phi(A_2,\ldots, A_k)=L_1+\sum_{i=2}^k\Gamma_{A_i}
\]
where for $1\leq i\leq k$, we take $L_i\in\bP(F_{\beta_i}^{\beta_{i-1}})$, $L_i^{\perp}:=\left(\tau_i^{\perp}\right)_{L_i}$ is a complement for $L_i$ in $F_{\beta_i}^{\beta_{i-1}}$ and $A_i\in \Hom(L_i,\sum_{j=1}^{i-1}L_j^{\perp})$ are linear maps, $i\geq 2$. To make the definition more symmetric let $A_1:L_1\ra \{0\}$. Then
\[\Phi((A_i)_{1\leq i\leq k})=\Phi\left((L_i,A_i)_{1\leq i\leq k}\right)=\sum_{i=1}^k\Gamma_{A_i}.
\]
\begin{lem} The map $\Phi$ is well-defined.
\end{lem}
\begin{proof} Since all graphs are $1$-dimensional, we need to check that $\Gamma_{A_{i+1}}\not\subset\sum_{j=1}^i\Gamma_{A_{j}}$ and this is a consequence of the fact that $\sum_{j=1}^i\Gamma_{A_j}$ is a subset of $F_{\beta_i}$ while $L_{i+1}\cap F_{\beta_i}=\{0\}$.
\end{proof}
\begin{theorem} \label{thm1}\begin{itemize}
\item[(i)] The map $\Phi:\bH\ra \Gr_k(E)$ is injective with image $V_{\beta}$ and makes the obvious diagram commutative together with the projections of $\bH$ and $V_{\beta}$ onto $\bP$.
\item[(ii)] The map $\Phi:\bH\ra V_{\beta}$ is a fiber bundle diffeomorphism over $\bP$.
\end{itemize}
\end{theorem}
\begin{proof} (i) Let us see first that $\Imag \Phi\subset V_{\beta}$. One needs to understand $\Phi(A_1,\ldots, A_k)\cap F_{\beta_i}$. Take $v_i\in L_i$ such that
\[ \sum_{i=1}^kv_i+\sum_{i=1}^kA_iv_i=w\in F_{\beta_i}
\]
We infer
\[F_{\beta_i}\oplus \sum_{j=i+1}^kL_{j}^{\perp}\ni w-\sum_{j=1}^iv_j-\sum_{j=1}^kA_jv_j=\sum_{j=i+1}^{k}v_j\in \sum_{j=i+1}^kL_j.
\]
and therefore $v_j=0$ for all $j\geq i+1$ and $w\in \sum_{j=1}^i\Gamma_{A_i}$.  We therefore  get
\[\Phi(A_1,\ldots,A_k)\cap F_{\beta_i}=\sum_{j=1}^i\Gamma_{A_i}
\]
as the other inclusion is obvious and the latter space has dimension $i$ as a direct sum of $1$-dimensional spaces.

Take now $\Phi((L_i,A_i)_{1\leq i \leq k})=\Phi((L_i',A_i')_{1\leq i\leq k})$. On one hand
\[P_{F_{\beta_i}^{\beta_{i-1}}}(\Phi((L_i,A_i)_{1\leq i\leq k})\cap F_{\beta_i})=P_{F_{\beta_i}^{\beta_{i-1}}}\left(\sum_{j=1}^i\Gamma_{A_i}\right)=P_{F_{\beta_i}^{\beta_{i-1}}}(\Gamma_{A_i})=L_i
\]
We therefore infer that $L_i=L_i'$, $i=1,\ldots,k$. This also proves that composing $\Phi$ and the projection $V_{\beta}\ra \bP$ gives the projection $\bH\ra \bP$.

Fix $v_i\in L_i$, $i=1,\ldots,k$. Then there exist $v_i'\in L_i$, $1\leq i\leq k$ such that:
\[\sum_{i=1}^kv_i+\sum_{i=1}^kA_iv_i=\sum_{i=1}^kv_i'+\sum_{i=1}^kA_i'v_i'
\]
Separate first the $\sum_{i=1}^kL_i$  and $\sum_{i=1}^kL_i^{\perp}$ components in order to get that $\sum v_i=\sum v_i'$ and since the $L_i$ components are linearly independent deduce that $v_i=v_i'$, $1\leq i\leq k$. Take now the $L_1^{\perp}$, $L_2^\perp$,\ldots $L_k^{\perp}$ components
\[\left\{\begin{array}{ccc}
\sum_{j=2}^k(A_j^1-(A_j')^1)v_j&=&0\\
\sum_{j=3}^k(A_j^2-(A_j')^2)v_j&=&0\\
\ldots\\
(A_k^{k-1}-(A_k')^{k-1})v_{k}&=&0
\end{array}\right.
\]
where $A_i^j$, $(A_i')^{j}$ are the $L_j^{\perp}$ components of $A_i$, $A_i'$. As this holds for all $v_i\in L_i$ one has $A_i=A_i'$, $i=2,\ldots, k$. Therefore $\Phi$ is injective.

We show now that $V_{\beta}\subset \Imag \Phi$. Let $L\in V_{\beta}$ and put $L_i:=P_{F_{\beta_i}^{\beta_{i-1}}}(L\cap F_{\beta_i})$. We show first that there exist $A_{i}':L_i\ra F_{\beta_{i-1}}$ such that
\[L\cap F_{\beta_i}=L_1+\sum_{j=2}^i\Gamma_{A_j'}
\]  
This is trivially true for $i=1$ with $A_1'\equiv 0$. Assume it is true for $i$. Since $\dim{L\cap F_{\beta_j}}=j$ we get that
\[L\cap F_{\beta_{i+1}}=L\cap F_{\beta_i}+L',\qquad L'\subset F_{
\beta_{i+1}},~  L'\cap F_{\beta_i}=\{0\}
\]
where $\dim{L'}=1$. The space $L'$ is just a complement of $F_{\beta_i}\cap L$ inside $L\cap F_{\beta_{i+1}}$. Clearly $P_{F_{\beta_{i+1}}^{\beta_{i}}}(L')=L_{i+1}$ and therefore $L'\subset L_{i+1}+F_{\beta_i}$. Since $L'\cap F_{\beta_i}=\{0\}$, there exists $A_{i+1}':L_{i+1}\ra F_{\beta_i}$ such that $L'=\Gamma_{A_{i+1}'}$. Hence
\[ L\cap F_{\beta_{i+1}}=\sum_{j=1}^{i+1}\Gamma_{A_j'}
\]
From $A_i'$ we obtain $A_i$ as follows. Decompose for $i\geq 1$, $A_{i+1}'=(B_{i+1},C_{i+1})$ with $B_{i+1}:L_{i+1}\ra \sum_{j=1}^i L_j$ and $C_{i+1}:L_{i+1}\ra \sum_{j=1}^i L_j^{\perp}$. Then $v\in L=\sum_{i=1}^k \Gamma_{A_i'}$ can be written as
\begin{equation}\label{eq411}v=v_1'+\sum_{j=2}^k(v_j'+B_jv_j'+C_jv_j')
\end{equation}
Set up now the following linear system of equations $v_i,v_i'\in L_i$:
\begin{equation}\label{eq5}\left(\begin{array}{c}v_1\\ v_2\\\ldots \\v_k \end{array}\right)=\left(\begin{array}{ccccc}1&  B_2^1 & B_3^1& \ldots & B_k^1\\
0&1 & B_3^2 &\ldots & B_k^2\\
&&\ldots &&\\
0& 0&0 &\ldots &  1 \end{array}\right) \left(\begin{array}{c}v_1'\\ v_2'\\\ldots\\  v_k' \end{array}\right)
\end{equation}
where the linear maps $B_i^j$ which take values in $L_j$ for $1\leq j<i$ are the components of $B_i$.

Clearly the system has a unique solution with $v_k=v_k'$ and 
\[v_i':=v_i-\sum_{i<j_1<\ldots <j_p\leq k}B^i_{j_1}B^{j_1}_{j_2}\ldots B_{j_p}^{j_{p-1}} v_{j_p},\quad i=1,\ldots k-1\]
Summing up the columns of $(\ref{eq5})$ gives $\sum_{j=1}^kv_j=v_1'+\sum_{j=2}^kv_j'+B_j'v_j'$. It follows that in (\ref{eq411}) one has
\[v=v_1'+\sum_{j=2}^kv_j'+B_jv_j'+C_jv_j'=v_1+\sum_{j=2}^kA_jv_j
\]
where $A_j:L_i\ra \sum_{p=1}^{j-1}L_p^{\perp}$ is a multilinear combination of $C_j$ and of the components of $B_j$, $j>i$.\\

(ii) Consequence of part (i).
\end{proof}
We produce now a conjugate Schubert variety $\overline{V^*_{\beta}}$ such that the corresponding open subset  $V_{\beta}^*$ intersects $V_{\beta}$ transversally along $\bP$. Moreover $V^*_{\beta}$ fibers over $\bP$ just like $V_{\beta}$. Define
\[V_{\beta}^*:=\{L\in \Gr_k(E)~|~\dim{L\cap G^{\beta_i}}=k-i,~~ i=0,\ldots k\}.
\] 
while $\overline{V_{\beta}^*}$ replaces the  dimensional condition by $\geq$.
\begin{rem} This $V_{\beta}^*$ is the opposite Schubert variety as defined in \cite{KrLa}, however not for the increasing sequence $\beta$  but for the increasing sequence $\gamma$ where $\gamma_i:=\beta_{i-1}+1$.
\end{rem}
\begin{prop} The set $V_{\beta}^*$ is a submanifold of $\Gr_k(E)$ of codimension $\sum_{i=1}^{k-1}\beta_i-i$.
\end{prop}
\begin{proof} Let $H_*:=G^{n-*}$ be the increasing flag obtained from $G^*$. Then
\[V^*_{\beta}=\{L~|~\dim{L\cap H_{n-\beta_i}}=k-i,~i=0,\ldots k\}=\{L~|~\dim{L\cap H_{n-\beta_{k-i}}}=i,~i=0,\ldots k\}
\]
Notice that $V_{\beta}^*$ is the open subset of the (regular) Schubert variety 
\begin{equation}\label{eq6}\widehat{V_{\beta}^*}=\{L~|~\dim{L\cap H_{n-\beta_{k-i}}}=i,~i=1,\ldots, k\}
\end{equation}
defined by the (open) condition $L\cap H_{n-\beta_k}=\{0\}$. 
\end{proof}
Notice that $V_{\beta}^*$ comes also with a projection to $\bP$:
\[L\ra \left(P_{F_{\beta_1}}(L\cap G^{\beta_0}),\ldots,P_{F_{\beta_k}^{\beta_{k-1}}}(L\cap G^{\beta_{k-1}})\right)
\]
where the projection uses the decomposition $G^{\beta_{i-1}}=G^{\beta_i}\oplus F_{\beta_i}^{\beta_{i-1}}$, $i=1,\ldots, k$.

Let now
\[\bH^*:=\prod_{i=1}^{k}\Hom\left(\tau_i,\sum_{j=i+1}^{k+1}\tau_{j}^{\perp}\right)\]
where as before $\tau_i\ra \bP\left(F_{\beta_i}^{\beta_{i-1}}\right)$ is the tautological bundle and $\tau_{k+1}^{\perp}:=G^{\beta_k}$ is the trivial vector bundle of dimension $n-\beta_k$.

\begin{theorem} \label{thm2} The map $\Phi^*:\bH^*\ra \Gr_k(E)$ defined by $\Phi^*((L_i,A_i)_{1\leq i\leq k})=\sum_{i=1}^k\Gamma_{A_i}$ is injective with image $V^*_{\beta}$. Moreover when restricting the codomain to $V_{\beta}^*$, $\Phi^*$ is a fiber bundle diffeomorphism with respect to the projections onto $\bP$.
\end{theorem}
\begin{proof} The map is well-defined and the rest follow closely the proof of Theorem \ref{thm1}. For example
\[\left(\sum_{j=1}^k\Gamma_{A_j}\right)\cap G^{\beta_i}=\sum_{j=i+1}^k\Gamma_{A_i},\quad i\leq k-1
\]
while for $i=k$, $\left(\sum_{j=1}^k\Gamma_{A_j}\right)\cap G^{\beta_k}=\{0\}$. Hence the image lands in $V_{\beta}^*$.
\end{proof}
\begin{rem} The relation between $V_{\beta}^*$ and $\widehat{V_{\beta}^*}$ from  (\ref{eq6}) can be understood as follows. Notice that $\widehat{V_{\beta}^*}$ fibers over
\[\bP'=\bP(F_{\beta_1})\times \ldots \bP(F_{\beta_{k-1}}^{\beta_{k-2}})\times \bP(G^{\beta_{k-1}})
\] as a direct application of Theorem \ref{thm1}. Note that $V_{\beta}^*$ is the open subset in $\widehat{V_{\beta}^*}$ which is the preimage via the projection onto $\bP'$ of the open set 
\[(\bP')^{\circ}=\{(\ell_1,\ldots,\ell_k)~|~\ell_k\cap G^{\beta_k}=\{0\}\}
\]
The  set $\bP(G^{\beta_{k-1}})^{\circ}=\{\ell_k~|~\ell_k\cap  G^{\beta_k}=\{0\}\}$ fibers over $\bP(F_{\beta_k}^{\beta_{k-1}})$ via $\ell_k\ra \bP_{F_{\beta_k}^{\beta_{k-1}}}(\ell_k)$. The projection map to $\bP$ is the composition of the projection to $\bP'$ followed by the projection $\bP(G^{\beta_{k-1}})^{\circ}\ra \bP(F_{\beta_k}^{\beta_{k-1}})$. 
\end{rem}
\begin{prop} The Schubert varieties $V_{\beta}$ and $V_{\beta}^*$ intersect transversely along $\bP$. This stays true about the varieties $\overline{V_{\beta}}$ and $\overline{V_{\beta}^*}$.
\end{prop}
\begin{proof} This is standard. One way to see it is to notice that the Schubert varieties $V_{\beta}$ and $V_{\beta}^*$ are the stable and unstable manifolds of the critical manifold $\bP$ for  a certain Morse-Bott function on $\Gr_k(E)$. We  give  a self-contained proof. Take $(L_1,\ldots, L_k)\in \bP$ and use the following chart\footnote{The word chart in this context is used as a shorthand  for the map between a vector space of morphisms and an open subset of the Grassmannian, flag manifold etc. which takes a morphism to its graph.} of $\Gr_k(E)$:
\[ W=\Hom\left(\sum_{i=1}^kL_i,\sum_{i=1}^{k}L_i^{\perp}+G^{\beta_k}\right)
\]
for some  complements $L_i^{\perp}$ in $F_{\beta_i}^{\beta_{i-1}}$, $i=1,\ldots k$. Let $T\in W$ and denote by $B_i:L_i\ra \sum_{j=1}^{i-1}L_j^{\perp}$, $i\geq 2$,  $A_i:L_i\ra L_i^{\perp}$, $C_i:L_i\ra\sum_{j=i+1}^{k+1}L_j^{\perp}$, $1\leq i\leq k$ the components of $T$ where $L_{k+1}^{\perp}:=G^{\beta_k}$. Notice that $T\in W$ implies that $\dim{\Gamma_T\cap F_{\beta_i}}\leq i$ since for every $i$, $\Gamma_T$ will not intersect $\sum_{j=1}^iL_j^{\perp}$ which is a codimension $i$ subspace of $F_{\beta_i}$.

Then $\Gamma_T\in V_{\beta}$ implies that $C_i\equiv 0$ for all $1\leq i\leq k$. This is done by induction on $i$. Similarly $\dim{\Gamma_T\cap G^{\beta_i}}\leq k-i$ and $\Gamma_T\in V_{\beta}^*$ implies that $B_i\equiv 0$.

Hence in the chart $W$, $V_{\beta}$ intersects $V_{\beta}^*$ only for those $T$ for which $B_i\equiv 0$ and $C_i\equiv 0$. This means that $\Gamma_T\in \bP$ and the intersection is transversal. Since we already know from the proofs of Theorems \ref{thm1} and \ref{thm2} that $V_{\beta}$ and $V_{\beta}^*$ can be covered with charts of type $W$ when we vary $(L_1,\ldots, L_k)$ we get the first claim. 

The second claim follows from the first claim combined with the fact that $\dim{L\cap F_{\beta_i}}>i$ implies that $\dim{L\cap G^{\beta_i}}<k-i$ for any $i$. Hence $\overline{V_{\beta}}\cap\overline{V_{\beta}^*}=V_{\beta}\cap V_{\beta}^*$.
\end{proof}
\begin{rem} The identification of the stable and unstable manifolds with concrete bundles as in Theorems \ref{thm1}, \ref{thm2} only appears in the particular case of $k=2$, $\beta_2 = n$ in \cite{HL}.
\end{rem}

\section{The normal directions}
From the description of $V_{\beta}$ and $V_{\beta}^*$ as vector bundles over $\bP$ and as varieties that intersect transversally we see that a candidate for the normal bundle of $V_{\beta}$ is $\pi_1^*\bH^*$ where $\pi_1 : V_{\beta} \ra \bP$ is the projection described in the previous section.

In order to get an embedded resolution of $\overline{V}_{\beta}$ we will use a fiber $\bH_0^*$ of $\bH^*\ra \bP$ and deform the flag which defines $V_{\beta}$, $F_{\beta}: ~0\subset F_{\beta_1}\subset \ldots \subset F_{\beta_k}\subset E$ in the "directions" of this fiber producing thus an $k(n-k)-\sum_{i=1}^k (\beta_i-i)$
family of flags of the same type each of which rendering a Schubert variety of the same type as ${V_{\beta}}$. The union of all these Schubert varieties covers an open dense set of $\Gr_k(E)$. Next we will compactify the fiber $\bH^*_0$ but in the space of flags and finally we will desingularize this compactification.

We fix once and for all the lines $L_i\subset F_{\beta_i}^{\beta_{i-1}}$ and complements for them
\[ L_i\oplus L_i^{\perp}=F_{\beta_i}^{\beta_{i-1}}
\]
Let 
\[ \bH^*_0:=\prod_{i=1}^k\Hom\left(L_i,\sum_{j=i+1}^{k+1}L_j^{\perp}\right)
\]
where $L_{k+1}:=G^{\beta_k}$. The \emph{vector space} $\bH_0^*$ has dimension equal to the codimension of $V_{\beta}$.

For $1\leq \gamma_1 <\gamma_2<\ldots<\gamma_k \leq n$ let
\[{\Fl}_{\gamma_1,\ldots,\gamma_k}(E):=\{0\subset F_{\gamma_1}' \subset\ldots\subset F_{\gamma_k}'\subset E,\quad \dim{F_{\gamma_i}'} =\gamma_i\} \]
be the space of flags of length $k$ and dimensions $\gamma_1,\ldots \gamma_k$. 

Let
 \[\Psi:\bH_0^*\ra {\Fl}_{\beta_1,\ldots,\beta_k}(E),\qquad (A_1,\ldots,A_k)\ra \left(\ldots,\sum_{j=1}^i\Gamma_{A_j}+\sum_{j=1}^iL_j^{\perp},\ldots \right)
 \]
 Notice that
 \begin{itemize}
 \item[(i)]  $\Gamma_{A_i}\not\subset \sum_{j=1}^{i-1}\Gamma_{A_j}+\sum_{j=1}^{i}(L_j^{\perp})$ since $\Gamma_{A_i}\subset L_i\oplus \sum_{j=i+1}^{k+1}(L_j^{\perp})$ and therefore the spaces in $\Imag \Psi$ have indeed dimensions $\beta_1,\ldots,\beta_k$.
 \item[(ii)] $\Psi(0)=( F_{\beta_1}, F_{\beta_2},\ldots F_{\beta_k})$;
 \item[(iii)] $\Psi$ is an embedding.
 \end{itemize}

Consider next the set
\[\mathcal{F}(E)\subset {\Fl}_{1,2,\ldots, k}(E)\times {\Fl}_{\beta_1,\ldots, \beta_k}(E),\quad \mathcal{F}(E)=\{(V_*,F_*')~|~V_i\subset F_{\beta_i}'\}
\]
It is standard that $\mathcal{F}(E)$ is a complex manifold and the projection onto the second coordinate induces the structure of a fiber bundle. The fiber of this fiber bundle over $F_*'=F_*$  when is projected onto $\Gr_k(E)$ via $V_*\ra V_k$ is the well-known Kempf-Laksov resolution of $\overline{V_{\beta}}$. In particular, it has dimension equal to $\dim{V_{\beta}}$.

The complex manifold 
\[\Psi^*\mathcal{F}(E)=\left\{(A_*,V_*)\in\bH_0^*\times{\Fl}_{1,2,\ldots,k}(E)~|~V_i\subset \sum_{j=1}^i\left(\Gamma_{A_j}+L_j^{\perp}\right)\right\}
\]
satisfies the following property.
\begin{theorem} \label{thm41}The map
\begin{equation}\label{eq7} \Psi^*\mathcal{F}(E)\ra \Gr_k(E),\qquad (A_*,V_*)\ra V_k
\end{equation}
contains in its image the chart $W=\Hom\left(\sum_{i=1}^k L_i,\sum_{i=1}^{k+1}L_i^{\perp}\right)$, an open dense set of $\Gr_k(E)$. Moreover the restriction of this map to $[\Psi(0)]^*\mathcal{F}(E)$ is the Kempf-Laksov resolution of $\overline{V_{\beta}}$.
\end{theorem}
\begin{proof} The last statement is obvious. We will only sketch the proof of the first statement.

 Notice that if $V_k$ is in the image of (\ref{eq7}) then  there exist a flag $F_{\beta_i}':=\sum_{j=1}^i(\Gamma_{A_j}+L_j^{\perp})\in\Imag \Psi$   such that $\dim{V_k\cap F_{\beta_i}'}\geq i$ for all $i$. On the other hand, if $T\in W$  then $\dim{\Gamma_T\cap F_{\beta_i}'}\leq i$ as $\sum_{j=1}^iL_j^{\perp}$ has codimension $i$ in $F_{\beta_i}'$ and $\Gamma_T\cap \sum_{j=1}^iL_j^{\perp}=\{0\}$. 

One then proves (by induction on $i$ and separation into relevant components) that the conditions $\dim{\Gamma_T\cap F_{\beta_i}'}=i$ imply that $A_i=B_i$ where $B_i=P(T\bigr|_{L_i})$, $P$ standing for the projection onto $\sum_{j=i+1}^{k+1}L_j^{\perp}$.

Then define $V_i:=\Gamma_T\cap F_{\beta_i}'$.

Hence if   $T\in W$   then there exists a \emph{unique} flag $F_{\beta}'\in\Imag \Psi$ and unique $V_*$ such that $(A_*,V_*)\ra \Gamma_T$.
\end{proof}
\begin{rem} The image of the map (\ref{eq7}) is definitely not contained in the open set $W$, since it contains $\overline{V_{\beta}}$ and  we know that $\overline{V_{\beta}}\cap W\subset V_{\beta}$.
\end{rem}
\begin{rem}
Since $\bH^*_0$ is not compact we need to look for a compactification in order to build the embedded resolution. The naive approach of taking
\[\prod_{i=1}^k\bP\left(L_i+\sum_{j={i+1}}^{k+1}L_j^{\perp}\right)
\]
does not work since the map $\Psi$ does not extend to this compactification.
\end{rem}

\section{The compactification of $\bH^*_0$}
We use the same notation as in the previous section. The main observation that leads to the compactification of $\bH^*_0$ is the following quite trivial relation. For $j\leq i$
\begin{equation}\label{eq8}\Gamma_{A_j}+\sum_{p=1}^{i}L_p^{\perp}=\Gamma_{A_j^i}+\sum_{p=1}^iL_p^{\perp}
\end{equation}
where $A_j^i$ is the composition of $A_j:L_j\ra \sum_{p=j+1}^{k+1}L_p^{\perp}$ with the natural projection 
\[   \sum_{p=j+1}^{k+1}L_p^{\perp}\ra\sum_{p=i+1}^{k+1}L_p^{\perp} 
\]

In particular, $A_i^i=A_i$. 

We conclude from (\ref{eq8}) that
\[F_{\beta_i}':=\sum_{p=1}^i\Gamma_{A_i}+\sum_{p=1}^iL_p^{\perp}=\sum_{p=1}^i\Gamma_{A_p^i}+\sum_{p=1}^iL_p^{\perp}.
\]
Hence we can define
\begin{equation}\label{eq56} B_i:\sum_{j=1}^iL_j\ra\sum_{j=i+1}^{k+1}L_j^{\perp},\qquad B_i\bigr|_{L_j}:=A_j^i
\end{equation}
and we obtain that
\begin{equation}\label{eq10} \Gamma_{B_i}+\sum_{p=1}^iL_p^{\perp}=F_{\beta_i}'.
\end{equation}

\begin{lem} \label{lem1} The following relation holds
\[ \Gamma_{B_{i}}\subset \Gamma_{B_{i+1}}+L_{i+1}^{\perp}
\]
\end{lem}
\begin{proof} For $j\leq i$, use the relation $A_j^i(v)=A_{j}^{i+1}(v)+A_j^{L_{i+1}^{\perp}}(v)$ for all $v\in L_j$, where $A_j^{L_{i+1}^{\perp}}$ is the ${L_{i+1}^{\perp}}$ component of $A_j$.
\end{proof}
The $k$-tuple $(\Gamma_{B_1},\ldots,\Gamma_{B_k})$ together with the previous relations suggest introducing the following object.

First, for $1\leq j\leq i\leq k$ let
\begin{equation}\label{eq9} V_j^i:=\sum_{p=1}^jL_p+\sum_{p=i+1}^{k+1}L_p^{\perp}
\end{equation}
Let 
\[G:=\prod_{i=1}^k\Gr_i(V_i^i)\subset \prod_{i=1}^k\Gr_i(E)
\]
and
\[\mathcal{G}:=\{(\ell_1,\ldots,\ell_k)\in G~|~\ell_i\subset \ell_{i+1}+L_{i+1}^{\perp},~1\leq i\leq k-1\}.
\]

By Lemma \ref{lem1} we have that $(\Gamma_{B_1},\ldots,\Gamma_{B_k})\in\mathcal{G}$ hence we will be regarding $\bH_0^*$ as a subset. of $\mathcal{G}$ by letting $\ell_i:=\Gamma_{B_i}$.

Relation (\ref{eq10}) suggests the following extension of the map $\Psi$:
\begin{equation}\label{eq45} \widetilde{\Psi}: G\ra \Fl_{\beta_1,\ldots,\beta_k}(E),\quad \left((\ell_1,\ldots,\ell_k)\ra \left(\ell_1+L_1^{\perp},\ldots,\sum_{i=1}^k\ell_i+\sum_{i=1}^kL_i^{\perp}\right) \right)
\end{equation}
The map is well-defined because
\[V_i^i\cap \sum_{p=1}^iL_p^{\perp}=\{0\},\qquad \forall~i.
\]
We are really interested in $\widetilde{\Psi}\bigr|_{\mathcal{G}}$. But we  should get acquainted with $\mathcal{G}$ first. In general it is a singular variety as the following example shows.
\begin{example} Let $k=2$ and suppose $G^{\beta_2}\neq 0$. Then
 \[\mathcal{G}=\{(\ell_1,\ell_2)\in\bP(L_1+L_2^{\perp}+G^{\beta_2})\times \Gr_2(L_1+L_2+G^{\beta_2})~|~\ell_1\subset\ell_2+L_2^{\perp}\}\]
 Singularities appear when $\ell_1\subset L_2^{\perp}$ and $L_1\subset \ell_2\subset L_1+G^{\beta_2}$ for which the unique incidence relation is automatically satisfied. Take $\ell_2=L_1+\ell'$ with $\ell'\subset G^{\beta_2}$ one dimensional. A chart centered at $\ell_2\in\Gr_2(L_1+L_2+G^{\beta_2})$ can be taken to be:
 \[W_2:=\Hom(L_1\oplus \ell',L_2\oplus (\ell')^{\perp})\ni B=\left(\begin{array}{cc} B_1&B_2\\
 B_3 &B_4\end{array}\right)
 \]
 where $\ell'\oplus (\ell')^{\perp}=G^{\beta_2}$ while a chart at $\ell_1$ is $W_1:=\Hom(\ell_1,L_1\oplus \ell_1^{\perp}\oplus \ell'\oplus(\ell')^{\perp})$ where $\ell_1^{\perp}+\ell_1=L_2^{\perp}$. Let $A=(A_1,\ldots,A_4)\in W_1$. 
 
 Once one introduces bases, one can think of $A_1$, $A_3$, $B_1$, $B_2$ as numbers. 
 
 Then the incidence relation $\Gamma_A\subset \Gamma_B+L_2^{\perp}$ translates into the equalities
 \begin{equation} B_1A_1+B_2A_3\equiv 0
 \end{equation}
 \begin{equation}\label{eq413}A_4\equiv B_3A_1+B_4A_3.
 \end{equation}
 The relation $B_1A_1+B_2A_3\equiv 0$ is the equivalent of the hypersurface $xy+zw=0$ in $\bC^4$ and is the source of singularities in this particular example. 
 
\end{example}
\begin{prop} The set $\mathcal{G}$ is a compact algebraic subvariety of $G$ of dimension $\dim \bH_0^*$.
\end{prop}
\begin{proof} Theorem \ref{mth} below proves that $\mathcal{G}$ is the image of an analytic map defined on a compact complex manifold of dimension $\bH^*_0$. In particular $\mathcal{G}$ is a compact complex analytic variety of dimension at most $\dim{\bH^*_0}$. Since the ambient manifold is projective we get that $\mathcal{G}$ is algebraic.

We prove that it is of dimension exactly $\dim {\bH^*_0}$ by exhibiting an open dense set of regular points in this dimension. The open set is obtained as follows. Recall first the definition of $V_j^i$ in (\ref{eq9}). Then notice that for every fixed $i$ and every $1\leq j\leq k-i$ we have:
\begin{equation}\label{eq11}\ell_i\subset\left(\ell_{i+j}+\sum_{p=i+1}^{i+j}L_p^{\perp}\right)\cap V_i^i=\left(\ell_{i+j}\cap V_i^{i+j}\right)+\sum_{p=i+1}^{i+j}L_p^{\perp}
\end{equation}
Now, $V_{i}^{i+j}$ is a codimension $j$ space in $V^{i+j}_{i+j}$ and $\ell_{i+j}\subset V_{i+j}^{i+j}$ has dimension $i+j$ hence a condition of type
\[\dim{\ell_{i+j}\cap V_i^{i+j}}=i
\] 
defines an open set in $\Gr_{i+j}(V_{i+j}^{i+j})$.

Hence $U_a:=\{\ell\in \Gr_a(V_a^a)~|~\dim {\ell\cap V^a_i}=i,~\forall~i\leq a-1\}$ is open and non-empty as $\ell=\sum_{j=1}^aL_j\in U_a$  and therefore 
\begin{equation} \label{eq24}U:=\prod_{a=1}^k U_a ~\mbox{ is an open set in } G.\end{equation}

Relations (\ref{eq11}) imply that the intersection $\mathcal{G}\cap U$ can be described as the set $(\ell_1,\ldots \ell_k)\in G$ such that each $\ell_i$ with $i\leq k-1$ satisfies two conditions:
\begin{itemize}
\item[(i)] $\ell_i\in \Gr_i(\ell_{i+1}\cap V_i^{i+1}+L_{i+1}^{\perp})\subset \Gr_i(V_i^i)$;
\item[(ii)] $\ell_i\in U_i\subset \Gr_i(V_i^i)$.
\end{itemize}
For $i=k$ there is only one restriction namely (ii).   We justify now why the set
 \[\{(\ell_{k-1},\ell_k)~|~\ell_{k-1}\in U_{k-1}, ~\ell_k\in U_k, ~\ell_{k-1}\subset \ell_k+L_k^{\perp}\}\] 
is a fiber bundle over $U_k$. Notice that for $i<k-1$ the relation $\ell_{k-1}\subset \ell_{k}\cap V^k_{k-1}+L_k^{\perp}$  implies:
\begin{equation} \label{eq12} \ell_{k-1}\cap V_i^{k-1}=\ell_{k-1}\cap [\left(\ell_k\cap V^k_{k-1}+L_k^{\perp}\right)\cap V^{k-1}_i]=\ell_{k-1}\cap \left(\ell_k\cap V_i^k+L_k^{\perp}\right)
\end{equation}
Collectively (\ref{eq12}) say that when  $\ell_k\in U_k$ is fixed then putting together (i) and (ii) we get that $\ell_{k-1}\in U_{k-1}\cap \Gr_{k-1}(\ell_k\cap V^{k}_{k-1}+L_k^{\perp})$ is the open condition in $\Gr_{k-1}(\ell_k\cap V^{k}_{k-1}+L_k^{\perp})$ described by the following $k-2$ (open) relations
\[ i=\dim \ell_{k-1}\cap Z_i(\ell_k)\quad (=\dim{\ell_{k-1}\cap V^{k}_{i}}),\qquad i<k-1
\]
where $Z_i(\ell_k)=\ell_k\cap V^k_{i}+L_k^{\perp}$. The relations are open because $Z_i(\ell_k)$ has codimension $k-1-i$ inside $\ell_k\cap V^{k}_{k-1}+L_k^{\perp}$. In particular the fiber type (of the projection onto $U_k$) over $\ell_k\in U_k$ does not depend on $\ell_k$. We hence get the fiber bundle structure. 

The same reasoning can be applied inductively to show that $\{(\ell_i,\ldots, \ell_k)~|~\ell_j\in U_j, \ell_j\subset \ell_{j+1}+L_{j+1}^{\perp}\}$ is a fiber bundle over $\{(\ell_{i+1},\ldots, \ell_k)~|~\ell_i\in U_i, \ell_j\subset \ell_{j+1}+L_{j+1}^{\perp}\}$ and conclude that $\mathcal{G}\cap U$ is a tower of fiber bundles with fibers of type open sets in $\Gr_i(\ell_i^{i+1}\cap V_i^{i+1}+L_{i+1}^{\perp})$. 

The dimension is then, remembering that $\dim {\ell_{i+1}\cap V^{i+1}_i}=i$:
\[\dim \Gr_k(V^k_k)+\sum_{i=1}^{k-1}\dim \Gr_i(\ell_i^{i+1}\cap V_i^{i+1}+L_{i+1}^{\perp})=k(n-\beta_k)+\sum_{i=1}^{k-1}i(\beta_{i+1}-\beta_i-1)=\dim \bH_0^*.
\]
\end{proof}
\begin{rem} It is not enough for a construction of an open dense set of regular points in $\mathcal{G}$ to consider open sets in $\Gr_i{(V_i^i)}$ induced just by the relation $\dim \ell_{i+1}\cap V^{i+1}_i=i$. One  needs that $\dim{\ell_{i+1}\cap V^{i+1}_j}=j$ for all $j\leq i$.
\end{rem}

We now construct a resolution for $\mathcal{G}$ via a bioriented flag.  Let $\widehat{G}$ be the $k\times k$ product of spaces $S_{ij}$ where
\[ S_{ij}= \left\{\begin{array}{cc}
 \pt & \mbox{if } ~ j>i\\
  \Gr_j(V^i_j) &\mbox{if } ~j\leq i
\end{array}\right.
\]
Shortly, ignoring the redundancy of point spaces:
\[\widehat{G}=\prod_{1\leq j\leq i\leq n}\Gr_j(V^i_j)
\]

We denote an element of $\widehat{G}$ by $(\ell_j^i)_{i,j}$ with $\dim{\ell_j^i}=j$. Define
\[\mwG:=\{(\ell_j^i)_{i,j}\in\widehat{G}~|~(I.)~\ell_j^i\subset \ell_{j+1}^i,~j\leq k-1; (II.)~\ell_j^i\subset \ell_j^{i+1}+L_{i+1}^{\perp},~i\leq k-1\}
\]
\begin{theorem}\label{mth}\begin{itemize}
\item[(a)] The space $\mwG$ is a compact manifold of dimension $\bH_0^*$.
\item[(b)] The projection $\pi:\mwG\ra G$, $\pi((\ell_j^i)_{i,j})=(\ell_i^i)_{1\leq i\leq k}$ is a surjection onto $\mathcal{G}$.
\item[(c)] The projection $\pi:\mwG\ra G$ of item (b) is a resolution of $\mathcal{G}$.
\end{itemize}
\end{theorem}
\begin{proof} At (a) one proves by induction starting with the last line ($i=k$) that $\mwG$ is a tower of fiber bundles. For the last line there is no $(II.)$ condition. Moreover $V^k_{j}\subset V^{k}_{j+1}$ and therefore the $(I.)$ condition together with $\ell^k_j\in \Gr_j(V^k_j)$ imply that the image of the projection onto the last line of $\mathcal{G}$ is the Kempf-Laksov resolution of the Schubert variety in $\Gr_k(V^k_k)$ defined by the incidence relations
\[\{L\in \Gr_k(V^k_k)~|~\dim{L\cap V^k_j}=j,~j=1,\ldots, k\}
\]
This is in fact an open dense set and therefore has the same dimension as $\Gr_k(V^k_k)$, i.e. $k(n-\beta_k)$.

Take now the line $i=k-1$ and notice that $V^{k-1}_j\cap (\ell^k_{j})+L_{k}^{\perp})=\ell^k_{j})+L_{k}^{\perp}$ for all $j\leq k-1$. This says that  the constrains and incidence relations on $\ell^{k-1}_j$ once $(\ell^k_j)_{1\leq j\leq k}$ are fixed, are equivalent with
\begin{itemize}
\item[(i)] $\ell^{k-1}_j\in \Gr_j((V^{k-1}_j)')$, where $(V^{k-1}_j)':=\ell^k_{j}+L_{k}^{\perp}$ is a space of dimension \linebreak $\dim{\ell^k_{j}}+\dim {L_k^{\perp}}=j+\beta_k-\beta_{k-1}-1$ since $\ell^k_{j}\cap L_k^{\perp}=\{0\}$;
\item[(ii)] $\ell^{k-1}_j\subset \ell^{k-1}_{j+1}$.
\end{itemize}
In other words, for fixed $(\ell^k_j)_{1\leq j\leq k}$ the tuples of spaces $(\ell^{k-1}_j)_{1\leq j\leq k-1}$ satisfy the conditions of the Kempf-Laksov resolution of the Schubert variety in $\Gr_{k-1}((V^{k-1}_{k-1})')$ defined by the incidence relations $\{L~|~\dim{L\cap V^{k-1}_j)'}= j,~j=1,\ldots, k-1\}$. One thus gets a manifold of dimension $(k-1)(\beta_{k}-\beta_{k-1}-1)$.

Similar reasoning works for the next step. Hence $\mwG$ has dimension 
\[k(n-\beta_k)+\sum_{j=1}^{k-1}j(\beta_{j+1}-\beta_j-1)=\dim H_0^*.\]

At $(b)$ notice first that since $\ell_i^i\subset \ell_i^{i+1}+L_{i+1}\subset \ell_{i+1}^{i+1}+L_{i+1}^{\perp}$ we get that the image of the projection $\pi$ is contained in $\mathcal{G}$. In order to prove surjectivity one uses again induction, at step one constructing the diagonal $(\ell_i^{i+1}))_{1\leq i\leq k-1}$, then the lower diagonal $(\ell_i^{i+2}))_{1\leq i\leq k-2}$, etc.

Let us start with $\ell_*\in\mathcal{G}$ such that $\ell_i\subset \ell_{i+1}+L_{i+1}^{\perp}$. We will show that from  the pair $(\ell_i,\ell_{i+1})$ we can choose an $\ell_i^{i+1}\subset \ell_{i+1}$ of dimension $i$ such that
\[ \ell_i\subset \ell_i^{i+1}+L_{i+1}.
\]
First we have that 
\begin{equation}\label{eq20} \ell_i\subset \left(\ell_{i+1}+L_{i+1}^{\perp}\right)\cap V_i^i=\left(\ell_{i+1}\cap\left(\sum_{j=1}^iL_j+\sum_{j=i+2}^{k+1}L_j^{\perp}\right)\right)+L_{i+1}^{\perp}=\ell_{i+1}\cap V_i^{i+1}+L_{i+1}^{\perp}
\end{equation}
On the other hand $\ell_{i+1}\in\Gr_{i+1}(V_i^{i+1}+L_{i+1})=\Gr_{i+1}(V^{i+1}_{i+1})$. Hence
\begin{equation}\label{eq21} \dim \ell_{i+1}\cap V_i^{i+1}\in \{i,i+1\}
\end{equation}
If the dimension of the intersection is $i$ then define 
\begin{equation} \label{eq34} \ell_i^{i+1}:=\ell_{i+1}\cap V_i^{i+1}\end{equation} and by  (\ref{eq20}) we have that
\[ \ell_i\subset\ell_i^{i+1}+L_{i+1}^{\perp}  ~\mbox{ and }~ \ell_i^{i+1}\subset \ell_{i+1} ~\mbox{ is obvious.}
\]
If the dimension in (\ref{eq21}) is $i+1$ then $\ell_{i+1}\subset V_i^{i+1}$. Consider the projection $\ell_i':=P_{V_i^{i+1}}(\ell_i)$ with respect to the decomposition
\[ V_i^i= V_i^{i+1}\oplus L_{i+1}^{\perp}
\]
We have $\ell_i\subset \ell_i'+L_{i+1}^{\perp}$ and also
\begin{equation} \label{eq22}\ell_i\subset(\ell_i'+L_{i+1}^{\perp})\cap (\ell_{i+1}+L_{i+1}^{\perp})=(\ell_i'\cap \ell_{i+1})+L_{i+1}^{\perp}.
\end{equation}
But $\ell_i'\cap \ell_{i+1}\subset V_i^{i+1}$ and therefore $\ell_i'\cap \ell_{i+1}\supset P_{V_i^{i+1}}=\ell_i'$. Since $\dim{\ell_i'}<\dim{\ell_{i+1}}$ we get $\ell_i'\subset \ell_{i+1}$. Choose now $\ell_i^{i+1}\subset \ell_{i+1}$ to be any subspace of dimension $i$ which contains $\ell_i$. Then since $\ell_{i+1}\subset V_i^{i+1}$ one has $\ell_i^{i+1}\in\Gr_{i}(V_i^{i+1})$ and because of (\ref{eq22}) the condition $\ell_i\subset \ell_i^{i+1} +L_{i+1}^{\perp}$ is also fulfilled.

Since $\ell_i^{i+1}\subset\ell_{i+1}\subset \ell_{i+1}^{i+2}+L_{i+2}^{\perp}$ we can use the $(k-1)$-tuple $(\ell_i^{i+1})_{1\leq i\leq k-1}$ to determine $\ell_i^{i+2}\in\Gr_i(V_i^{i+2})$ such that $\ell_i^{i+1}\subset \ell_i^{i+2}+L_{i+2}^{\perp}$ and $\ell_i^{i+2}\subset \ell_{i+1}^{i+2}$. And so on.

Finally, for (c) we need only check that over the set $\mathcal{G}\cap U$ the projection $\pi$ is one-to-one. We use here the open set $U$ which appears in (\ref{eq24}). Notice that given $(\ell_i^i)_{1\leq i\leq k}\in\mathcal{G}\cap U$ the first diagonal below the main diagonal is uniquely determined by (\ref{eq34}). Indeed we are looking for an $i$-dimensional space $\ell_i^{i+1}$ such that $\ell_i^{i+1}\subset \ell_{i+1}^{i+1}$, $\ell_i^{i+1}\subset V_i^{i+1}$ and $\dim{\ell_{i+1}^{i+1}\cap V_i^{i+1}}=i$. 

We go to the second diagonal and use that $\ell_i^{i+1}\subset\ell_{i+1}^{i+1}\subset\ell_{i+1}^{i+2}+L_{i+2}^{\perp}$. Hence
\[\ell_{i}^{i+1}\subset(\ell_{i+1}^{i+2}+L_{i+2}^{\perp})\cap V_i^{i+1}=(\ell_{i+1}^{i+2}\cap V_i^{i+1})+L_{i+2}^{\perp}=
\]
\[=(\ell_{i+1}^{i+2}\cap V_{i+1}^{i+2}\cap V_i^{i+1})+L_{i+2}^{\perp}=(\ell_{i+1}^{i+2}\cap V_i^{i+2})+L_{i+2}^{\perp}=
\]
\[=[(\ell_{i+2}^{i+2}\cap V_{i+1}^{i+2})\cap V_i^{i+2})]+L_{i+2}^{\perp}=(\ell_{i+2}^{i+2}\cap V_i^{i+2})+L_{i+2}^{\perp}.
\]
Then the unique choice for $\ell_i^{i+2}$ which satisfies $\ell_i^{i+2}\subset \ell_{i+2}^{i+2}$, $\ell_i^{i+2}\subset V_i^{i+2}$ and $\ell_{i}^{i+1}\subset \ell_{i}^{i+2}+L_{i+2}^{\perp}$ is $\ell_i^{i+2}=\ell_{i+2}^{i+2}\cap V_i^{i+2}$ which has dimension $i$ by the open condition $\ell_*\in U$. One does the same for the other diagonals. Hence
\[ \ell_j^i=\ell_i^i\cap V^i_j,\qquad \forall j\leq i.
\]
\end{proof}
\section{The embedded resolution}

Recall the total family of Kempf-Laksov resolutions $\mathcal{F}(E)\ra \Fl_{\beta_1,\ldots,\beta_k}(E)$.

The composition of $\widetilde{\Psi}:\mathcal{G}\ra \Fl_{\beta_1,\ldots,\beta_k}(E)$ with the resolution $\pi:\mwG\ra \mathcal{G}$  is denoted $\widehat{\Psi}$. The fiber product of the maps
\[ \xymatrix{ & \mathcal{F}(E)\ar[d] \\
  \mwG\ar[r]^{\widehat{\Psi}\qquad} & \Fl_{\beta_1,\ldots,\beta_k}(E)}
\]
is a compact complex manifold of dimension $k(n-k)$ denoted :
\[\widehat{\Psi}^*\mathcal{F}(E):=\{(\ell_*^*,V_*)~|~V_i\subset \ell_i^i+\sum_{j=1}^iL_j^{\perp}, ~\forall~i=1,\ldots,k\}
\]
From the proof of item (c) of Theorem \ref{mth} we infer that there exists a special point  $o\in\mwG$ which  corresponds to the unique bioriented  flag $\ell_*^*$ such that 
\[\pi(\ell_*^*)=\left(\ldots,\sum_{j=1}^iL_j,\ldots\right)\in\mathcal{G}\cap U
\] 
which corresponds to $B_i\equiv 0$ or, equivalently $A_i\equiv 0$ for all $i$.
\begin{theorem}\label{mth2} The map $\widehat{\Psi}^*\mathcal{F}(E)\ra \Gr_k(E)$ 
\[ (\ell_*^*,V_*)\ra V_k
\]
is an embedded resolution of $\overline{V_{\beta}}$ which restricts on $[\widehat{\Psi}(o)]^*\mathcal{F}(E)$ to the Kempf-Laskov resolution of $\overline{V_{\beta}}$.
\end{theorem}
\begin{proof} It is  easy to see that there exists an embedding $\iota:\bH^*_0\hookrightarrow \mwG$ onto an open dense set that lifts $\Psi$, i.e.
\[ \widehat{\Psi}\circ\iota=\Psi.
\]
This is because given $(A_1,\ldots,A_k)\in\bH^*_0$ then the $k$-tuple $(\Gamma_{B_1},\ldots,\Gamma_{B_k})\in\mathcal{G}$ obtained via (\ref{eq56}) actually belongs to the open set $U$ defined in (\ref{eq24}) i.e.
\[\dim{\Gamma_{B_i}\cap V^i_j}=i-j
\]
since $\Gamma_{B_i}\cap V^i_j=\Gamma_{B_i\bigr|_{\sum_{p=1}^j}L_p}$.  Now $U$ is a set over which the projection $\pi:\mwG\ra\mathcal{G}$ is a biholomorphism.

We also conclude from here that the map $\widehat{\Psi}^*\mathcal{F}(E)\ra \Gr_k(E)$ is an extension of the one from Theorem \ref{thm41} and since we are dealing with analytic maps between projective manifolds of the same dimension we get that it is surjective.

Finally, we prove item (c) of Definition \ref{defo1}. Consider the open chart in $\Gr_k(E)$:
\[ W=\Hom\left(\sum_{i=1}^kL_i,\sum_{i=1}^{k+1}L_i^{\perp}\right)
\]
Just as in the proof of Theorem \ref{thm41}, $T\in W$ implies that 
\[\dim{\Gamma_T\cap F_{\beta_i}}'=\dim{\Gamma_T\cap \left(\ell_i^i+\sum_{j=1}^iL_j^{\perp}\right)}\leq i,\qquad\forall~i\] for any flag $F_*'=\widehat{\Psi}(\ell_*^*)\in \Imag\widehat{\Psi}$.
But $\Gamma_T=V_k$ for some $(\ell_*^*,V_*)\in \widehat\Psi^*\mathcal{F}(E)$. Hence $\Gamma_T=V_k\supset V_i$ and by the incidence relations $V_i\subset \ell_i^i+\sum_{j=1}^{i}L_j^{\perp}$ it follows that
\[ \dim{\Gamma_T\cap \left(\ell_i^i+\sum_{j=1}^iL_j^{\perp}\right)} \geq i
\]
Hence equality holds and this implies on one hand that  $V_i=\Gamma_T\cap \left(\ell_i^i+\sum_{j=1}^iL_j^{\perp}\right)$ and that $\ell_*^*$ is uniquely determined. We prove the second claim. Notice that
\begin{equation}\label{eq98}\ell_i^i\cap \left(\sum_{j=i+1}^{k+1}L_{j}^{\perp}\right)=\{0\}.\end{equation} 
If (\ref{eq98}) were not to hold then we get a contradiction by considering $L$,  a complement of $\sum_{j=1}^{k+1}L_{j}^{\perp}$ (notice the range of the sum) inside $\ell_i^i+\sum_{j=1}^{k+1}L_{j}^{\perp}$ which will obviously have dimension smaller than $i=\dim{\ell_i^i}$ and the projection of $\Gamma_T\cap (\ell_i^i+\sum_{j=1}^iL_j^{\perp})$ inside $\ell_i^i+\sum_{j=1}^{k+1}L_{j}^{\perp}$ onto $L$. This projection is injective since $\Gamma_T\cap \sum_{j=1}^{k+1}L_{j}^{\perp}=\{0\}$ and we get $\dim{\Gamma_T\cap (\ell_i^i+\sum_{j=1}^iL_j^{\perp})}<i$.

Now (\ref{eq98}) implies (recall that $\ell_i^i\subset V_i^i$) that $\ell_i^i=\Gamma_{B_i}$ where $B_i\in \Hom\left(\sum_{j=1}^iL_j,\sum_{j=i+1}^{k+1}L_j^{\perp}\right)$. But then as we remarked earlier this means that the tuple $(\ell_i^i)_{1\leq i\leq k}$ belongs to $\mathcal{G}\cap U$ over which the projection $\pi:\mwG\ra \mathcal{G}$ is a biholomorphism.  

We conclude that the chart $W$ lies in the set over which the map $\widehat{\Psi}^*\mathcal{F}(E)\ra \Gr_k(E)$ is one-to-one. The intersection  $W\cap\overline{V_{\beta}}=V_{\beta}^{\circ}$ is an open dense set of $\overline{V_{\beta}}$, the Schubert cell of
\[F_{\beta_1-1}\subset F_{\beta_1}\subset\ldots \subset F_{\beta_2-1}\subset F_{\beta_2}\subset \ldots\subset F_{\beta_k-1}\subset F_{\beta_k}
\]
where
\[ F_{\beta_i-1}=\sum_{j=1}^{i-1}L_j+\sum_{j=1}^i L_j^{\perp}
\]
\[V_{\beta}^{\circ}=\{L\in\Gr_k(E)~|~\dim{L\cap F_{\beta_i}}=i,~~\dim{L\cap F_{\beta_i-1}}=i-1\}.
\]
We conclude that the only points that lie over $V_{\beta}^{\circ}$ in the resolution belong to the set $[\widehat{\Psi}(o)]^*\mathcal{F}(E)$.
\end{proof}
Rather than using the Kempf-Laksov resolution as the model for the desingularization  of the Schubert variety in Theorem \ref{mth2}, one could use other resolutions that are completely determined by the nodes $F_{\beta_1},\ldots, F_{\beta_k}$ like the small resolutions of Zelevinski \cite{Ze}.

\end{document}